\documentclass[pdftex]{amsart}
\usepackage[foot]{amsaddr}
\usepackage{latexsym}
\usepackage{amssymb,amsmath}
\usepackage{graphicx, xcolor}
\usepackage{multirow}
\usepackage{esint}
\usepackage{algorithmic}
\usepackage{algorithm}
\usepackage{here}
\usepackage{url}

\theoremstyle{plain}
\newtheorem{thm}{Theorem}[section]

\newtheorem{lem}[thm]{Lemma}

\theoremstyle{definition}

\theoremstyle{remark}
\newtheorem{rem}[thm]{Remark}

\newcommand{\red}[1]{{\color[rgb]{1.0, 0.0, 0.0}{#1}}}
\newcommand{\blue}[1]{{\color[rgb]{0.0,0.0,1.0}{#1}}}

\makeatletter
\def\@seccntformat#1{\csname the#1\endcsname.\quad}
\makeatother

\renewcommand{\theenumi}{\roman{enumi}}

\newcommand{\sgn}{\mathrm{sgn}}

\makeatletter
        
        \@addtoreset{equation}{section}
\makeatother

\makeatletter
        
        \@addtoreset{figure}{section}
\makeatother

\begin{document}

\def\sgn{\mathrm{sgn}}

\title{Diffusive method for multiple class data classification}
\author{T.~Aiki}
\address{Department of Mathematics, Faculty of Science, 
Japan Women's University, Tokyo, Japan}
\email{aikit@fc.jwu.ac.jp}
\author{H.~Yano}
\address{Division of Mathematical and Physical Sciences, Graduate School of Science, 
Japan Women's University, Tokyo, Japan}
\email{m1816099yh@ug.jwu.ac.jp}
\author{T.~Ohtsuka}
\address{
Department of Mathematics, Faculty of Science and Technology, Tokyo University of Science, Chiba, Japan}
\email{tohtsuka@rs.tus.ac.jp}
\subjclass[2020]{Primary~35K05, 68T10, Secondary~35K08, 35Q68}

\begin{abstract}
 A simple algorithm for multiple class data classification 
 is proposed, which is based on the infinit speed of propagation
 and the strong maximum principle for the heat equation.
 In this method, solutions of the heat equation whose initial data
 are the characteristic function of the training dataset
 are considered.
 The classification is established by the sign of
 the difference of two solutions in a very short time interval.
 A numerical algorithm for discrete training data
 in high-dimensional feature space is proposed.
 As an application of the proposed method,
 numerical results on the classification of handwritten digits
 database are presented in this paper.
\end{abstract}

\maketitle

\section{Introduction}

Data classification or pattern recognition 
is one of problems in machine learning.
See a book \cite{Mitchell:1997book} 
for details on the conventional theory of machine learning.
As a simple case, we consider the binary data classification,
which is the problem classify inputs into two categories,
either the datum belongs to a class or not.
From analytical viewpoint, this problem can be
regarded as a problem finding a function 
$f \colon \mathbb{R}^d \supset D \to \{ -1 , 1 \}$
to minimize the recognition error functional 
with given training data
\[
 S = \{ (x_\ell , y_\ell) \in \mathbb{R}^d \times \{ -1, 1 \} 
 \mid \ell = 1, 2, \ldots, N \}.
\]
In 1930s, Fisher \cite{Fisher:1950Book} proposed the theory of 
linear discriminants for this problem;
find a hyperplane
\[
 f(x) = w \cdot x + b
\]
for a vector $w \in \mathbb{R}^d$ and constant $b \in \mathbb{R}$
satisfying $y_\ell f(x_\ell) > 0$ for every $\ell = 1, \ldots, N$.
Rosenblatt \cite{Rosenblatt:Percept}
proposed an iterative method to obtain the linear discreminant
for linearly separable training data.
For improvement minimizing the recognition error, 
the idea of maximal margin classifier was proposed.
The maximal margin classifier by the geometrical margin
can be understood to find the hyperplane maximizing the
distance from both of the data
$\{ (x_\ell, y_\ell) \in S \mid y_\ell = 1 \}$
and 
$\{ (x_\ell, y_\ell) \in S \mid y_\ell = -1 \}$
subject to $y_\ell f(x_\ell) > 0$ for every $\ell = 1, \ldots, N$.
For linearly unseparable data, there are two direction 
to develop the linear discreminants.
One is the soft margin optimization, which extends 
the margin to that includes a few data results classification error.
The other is a kernel method, which extends the class of
discriminants to a space which has a basis of nonlinear functions,
and learn the discreminants by the theory of linear discriminants. 
The support vector machine 
due to Boser, Guyon and Vapnik \cite{Vapnik:SupportVectorMachine}
is one of the powerful options for the data classifications.
We refer to \cite{Vapnik:Book, cristianini_shawe-taylor_2000} 
for details of the theory of optimal separating hyperplane 
and its generalization for nonseparable case, or
the support vector machines and other kernel-based learning methods.
In summary of the above review, 
the key idea of the data classification
is to get a boundary of data class maximizing margin
from the training data.

Recently, 
a diffusive method for binary data classification 
is proposed by \cite{GGOU:2013CPAA},
which is not based on kernel method.
It is established by the diffusive sign
\[
 S_D [u_0] = \lim_{t \searrow 0} \mathrm{sgn} u(x,t)
\]
with a solution $u=u(x,t)$ of the Cauchy problem
\begin{align}
 \left\{
 \begin{array}{ll}
  u_t = \Delta u & \mbox{in} \ \mathbb{R}^d \times (0,\infty), \\
  u(\cdot, 0) = \chi_{A_1} - \chi_{A_2}
  & \mbox{on} \ \mathbb{R}^d
 \end{array}
 \right.
 \label{Cauchy heat}
\end{align}
for given training data $A_1, A_2 \subset \mathbb{R}^d$,
where
\begin{align*}
 \sgn (a) & =
 \left\{
 \begin{array}{ll}
  -1 & \mbox{if} \ a < 0, \\
  0 & \mbox{if} \ a = 0, \\
  1 & \mbox{otherwise},
 \end{array}
 \right. \\
 \chi_A (x)
 & =
 \left\{
 \begin{array}{ll}
  1 & \mbox{if} \ x \in A, \\
  0 & \mbox{otherwise}
 \end{array}
 \right. 
\end{align*}
for $A \subset \mathbb{R}^n$.
The unique bounded solution $u$ is known (see e.g., \cite{Widder:book})
to be represented by the Gaussian kernel $G_t$ of the form
\[
 u(x,t) = \int_{\mathbb{R}^d} G_t (x-y) u_0 (y) dy
 = (G_t * u_0) (x)
\]
with $G_t (x) = (4 \pi t)^{-d/2} \exp (-|x|^2/(4t))$.
The smoothing effect of the heat equation implies that
$u \in C^\infty (\mathbb{R}^d \times (0,\infty))$
(see e.g., \cite{Brezis:book}),
and the strong maximum principle implies that
the zero level set of $u$ for \eqref{Cauchy heat}
becomes analytic variety having locally finite 
$d-1$ dimensional Hausdorff measure
(see \cite{Protter-Weinberger:1967} for the details of
maximum principle).
Hence, the discriminant of this approach 
is proposed by the set where
$S_D [u_0] (x) = 1$ or $S_D [u_0] (x) = -1$.
At a practical stage, the discriminant is obtained as
\[
 P_{A_1}^h
 = \{ x \in \mathbb{R}^d \mid u(x,h) > 0 \}, \quad 
 P_{A_2}^h
 = \{ x \in \mathbb{R}^d \mid u(x,h) < 0 \}
\]
for sufficiently small $h > 0$.
The idea maximizing margin in the diffusive method 
is naturally established by the fundamental property 
of the heat equation.
According to the infinitely propagation speed
of the domain of dependence on the heat equation,
the discreminants by $S_D [u_0] (x) = \pm 1$ 
are characterized only by essential distance 
from the training data,
where ``essential'' means ignoring 
null subsets of training data.
Consequently, this approach leads to the nearest
neighbor classification by the essential distance.

The goal of this paper is to extend the diffusive method
to multiple class classification problem, that is,
find a classifier function 
$f \colon \mathbb{R}^d \to \{ 0, 1 \}^N$.
The key observation of the extention is that
we can divide the propagation of signs 
for the binary data problem into that of positive and
negative parts, i.e., 
\begin{align*}
 & u_t^+ = \Delta u^+, \quad 
 u_t^- = \Delta u^-
 & & \mbox{in} \ \mathbb{R}^d \times (0,T), \\ 
 & u^+ (\cdot,0) = \chi_{A_1}, 
 \quad u^- (\cdot, 0) = \chi_{A_2} 
 & & \mbox{on} \ \mathbb{R}^d.
\end{align*}
Clearly, the solution $u$ of \eqref{Cauchy heat}
is given as $u (x,t) = u^+ (x,t) - u^- (x,t)$, and thus
the diffusive method for binary data class
is rephrased by the diffusive sign of $u^+ (x,t) - u^- (x,t)$.
We now generalize this idea to the multiple class 
data classification;
let us consider the system of the heat equations
\begin{align}
 & \partial_t u_i = \Delta u_i & & \mbox{in} \ \mathbb{R}^d \times (0,T), 
 \label{intro: i-th eq} \\
 & u_i (\cdot, 0) = u_{0,i} := \chi_{A_i} & & \mbox{on} \ \mathbb{R}^d
 \label{intro: i-th init}
\end{align}
for a given training dataset $S = \bigcup_{i=1}^N A_i$,
where $A_i$ denotes the set of the data labeled by $i \in \{ 1, \ldots, N \}$.
Accordingly, the diffusive method for multiple class data classification 
can be established as
\[
 S_i = \{ x \in \mathbb{R}^d \mid S_D [u_{0,i} - u_{0,k}] (x) = 1 \ \mbox{for every} \
 k \neq i \}.
\]
As an application of the diffusive method for multiple class 
data classification,
we demonstrate our approach for the classification of the handwritten digits 
database provided by MNIST \cite{MNIST_Url,MNIST}.
Note that, at the practical stage,
datasets are often given by \emph{discrete} training data
in a \emph{high-deminsional feature space}.
In fact, the MNIST database of handwritten digits consist of
60,000 training and 10,000 test data,
which are given by $28 \times 28$-dimensional vectors 
(i.e., $d = 28 \times 28$).
To address such a discrete training data, 
we use the sum of heat kernels
centered at the training data.
For each $i \in \{ 1, \ldots, N \}$, we define
\begin{equation}
 \Psi_i (x,t) = \frac{1}{(4 \pi t)^{\frac{d}{2}}}
 \sum_{k=1}^{K_i} \exp \left( - \frac{|x-p_{k,i}|^2}{4t} \right).
 \label{intro: discreminants for discrete data1}
\end{equation}
Due to this notation, the classification of an input 
$x \in \mathbb{R}^d$ should be established by the diffusive sign 
of $\Phi_i (x,t) - \Phi_j (x,t)$, that is,
whether
\[
 \lim_{t \to 0} (\Psi_i (x,t) - \Psi_j (x,t))
\]
is positive or not for $j \neq i$.
When we compute $\Psi_i (x,t)$ for inputs 
$x \in \mathbb{R}^d \setminus S$ at a practical level,
numerical overflow or underflow often occurs since
$(4 \pi t)^{- d/2} \to \infty$ and
$\exp (- |x-p_{k,i}|^2 / (4t)) \to 0$ as $t \to +0$
although $(4 \pi t)^{- d/2} \exp (- |x-p_{k,i}|^2 / (4t)) \to 0$
as $t \to +0$
for every $p_{k,i} \in A_i$ ($k \in 1, \ldots, K_i$, $i=1, \ldots, N$).
To address this issue, 
we propose an algorithm in \S \ref{sec: 3.2}
finding a pair $i$ and a time $t_*$ such that
\begin{align*}
 & t_* = \inf \{ t_0 > 0 \mid \max_{1 \le \ell \le N} \Psi_\ell (x,t) < \varepsilon_0 
 \quad \mbox{for} \ t \in (0,t_0) \}, \\ 
 & (4 \pi t)^{\frac{d}{2}} \Psi_i (x,t_*)
 = \max_{1 \le \ell \le N} (4 \pi t)^{\frac{d}{2}} \Psi_\ell (x,t_*)
\end{align*}
for a sufficiently small $\varepsilon_0$ like as the machine epsilon.
Consequently, 
this idea also leads to the nearest neighbor classification
for multiple class training data,
or the Voronoi diagram if every $A_i$ is a singleton.

A simple approach to establish the nearest neighbor classification
is to compute the distance function from each dataset.
There are some algorithms: PDE approach \cite{Tsai_2002JCP},
fast sweeping method \cite{Zhao_2005MathComp}.
Our approach investigates which label of training data is 
nearest by using the solutions of the heat equations
instead of computing the distance functions.
Important feature of our approach is infinite speed of 
the propagation for the heat equation, which establishes
faster computation than computing distance functions
or the eikonal equation.
On the other hand, there are some methods using Gaussian
kernel in machine learning: kernel based learning with Gaussian
kernel, or linear classifier with radial basis functions
(see \cite{Vapnik:Book} for details of these methods).
However, the usage of the heat kernel in our approach
is different from them.

In this paper, we first present some mathematical analysis
on the diffusive sign in \S \ref{sec: mathematical framework}.
Note that considering \eqref{intro: i-th eq}--\eqref{intro: i-th init}
enables us to consider the situation such that
$u_i$ satisfies different diffusion equations from each other.
In this paper, we generalize the results of the diffusive sign
in \cite{GGOU:2013CPAA} for the situation such that
each heat equation of $u_i$ has different diffusion coefficient.
Moreover, we consider the property of $\lim_{t \to 0} \sgn (\Psi_i (x,t) - \Psi_j (x,t))$
for discrete dataset $A_i$ in \S \ref{sec: for discrete data}.
We next propose an algorithm of multiple class data classification,
in particular for discrete datasets in high-dimensional feature space,
in \S \ref{sec: algorithm}.
In \S \ref{numerical_results}, we present some numerical results
on the multiple class data classification,
in particular for the handwritten digits database by MNIST.

\section{Mathematical framework}
\label{sec: mathematical framework}

\subsection{For continuum data}

Consider the situation such that
training data are given as measurable sets 
$A_i \subset \mathbb{R}^d$ with $|A_i| \neq 0$
for $i=1, 2, \ldots, N$ and
$\bigcup_{i=1}^N \overline{A_i} \subsetneq \mathbb{R}^d$,
where $|A|$ denotes the Lebesgue measure 
of $A \subset \mathbb{R}^d$.
The goal of this section is to establish a diffusive method 
of $N$-class classification for continuum training data,
i.e., construct the classifier sets $P_{A_i}$
such that 
\begin{itemize}
 \item $\bigcup_{i=1}^N \overline{P_{A_i}} = \mathbb{R}^d$
       and $\bigcap_{i=1}^N P_{A_i} = \emptyset$,
 \item 
       $P_{A_i}$ indicates a class containing $A_i$.
\end{itemize}
The algorithm is as follows. 

\par
\bigskip
\noindent
\textbf{Diffusive method for $N$-class data classification (for continuum data).}
\begin{enumerate}
 \item[{\bf 1.}] Set 
	      \begin{equation}
	       \label{initial data}
	       u_{0,i} = c_i \chi_{A_i} (x)
	       = \left\{
	       \begin{array}{ll}
		c_i & \mbox{if} \ x \in A_i, \\
		0 & \mbox{otherwise}
	       \end{array}
	       \right.	       
	      \end{equation}
	      with a constant $c_i > 0$ for $i=1, \ldots, N$.
	      (Note that magnitude of $c_i$ plays no role
	      in this algorithm.
	      See Lemma \ref{characterization} for details.)
 \item[{\bf 2.}] Solve the Cauchy problem: 
	      \begin{align*}
	       & \partial_t u_i = a_i \Delta u_i
	       & & \mbox{in} \ \mathbb{R}^d \times (0,T), \\
	       & u_i (\cdot, 0) = u_{0,i}
	       & & \mbox{in} \ \mathbb{R}^d
	      \end{align*}
	      for $i=1, \ldots, N$, 
	      where $a_i > 0$ ($i=1, \ldots, N$) are constants.
 \item[{\bf 3.}] Set 
	      \[
	       P_{A_i} = \{ x \in \mathbb{R}^d \mid
	       S_D [u_{0,i} - u_{0,j}] (x) = 1 \ \mbox{for every} \ 
	       j \neq i \}
	      \]
	      for $i=1, \ldots, N$,
	      where
	      \[
	       S_D [u_{0,i} - u_{0,j}] (x) 
	       = \lim_{t \downarrow 0} \mathrm{sgn} (u_i (x,t) - u_j (x,t)).
	      \]
	      At a practical stage, fix a small parameter $0 < h \ll 1$,
	      and set
	      \[
	       P^h_{A_i} = \{ x \in \mathbb{R}^d \mid 
	       u_i (x,h) - u_j (x, h) > 0 \ \mbox{for every} \ 
	       j \neq i \}
	      \]
	      for $i=1, \ldots, N$.
\end{enumerate}

\par
\bigskip
\noindent
This algorithm proposes a kind of maximal (geometical) margin classifier.
We here present mathematical analysis on this algorithm.

First, we present a sufficient condition for initial data so that
$S_D [u_{0,j} - u_{0,k}] (x)$ is well-defined.
To clarify the situation, we now reduce the situation into 
a pair of solutions $u$ and $v$ for the heat equation.
Consider a pair of Cauchy problems
\begin{align}
 & u_t = a \Delta u, 
 & \mbox{in} & \ \mathbb{R}^d \times (0,T), 
 \label{positive part eq} \\
 & 
 v_t = b \Delta v
 & \mbox{in} & \ \mathbb{R}^d \times (0,T), 
 \label{negative part eq} \\
 & u (\cdot, 0) = u_0, \quad
 v (\cdot, 0) = v_0
 & \mbox{on} & \ \mathbb{R}^d,
 \label{pos-neg initial data}
\end{align}
where $a$ and $b$ are positive constants,
and the inital data $u_0, v_0 \in L^\infty (\mathbb{R}^d)$ 
are nonnegative functions.
For the solutions $u$ and $v$, we present a sufficient condition such that
\[
 S_D [u_0 - v_0] (\hat{x}) = \lim_{t \downarrow 0} \mathrm{sgn} (u(\hat{x},t) - v(\hat{x},t))
\]
is well-defined.
Note that $\lim_{t \downarrow 0} u(\hat{x},t) = u_0 (\hat{x})$ 
and $\lim_{t \downarrow 0} v(\hat{x},t)= v_0 (\hat{x})$ 
if $u_0$ and $v_0$ are continuous at $\hat{x}$.
Thus, we here consider the case when $u_0 (\hat{x}) = v_0 (\hat{x})$.

We now introduce a sign changing number of $f \colon \mathbb{R} \to \mathbb{R}$.
For an interval $I \subset \mathbb{R}$,
define 
\[
 Z_I [f] = \sup \left\{ m \ge 0 \; \middle| \; 
 \begin{aligned}
  & \mbox{there exists} \ \xi_0 < \xi_1 < \xi_2 < \cdots < \xi_m \ \mbox{such that} \\
  & \{ \xi_i \}_{i=0}^m \subset I, \
  \mbox{and} \ f(\xi_{i-1}) f(\xi_i) < 0 \ 
  \mbox{for} \ i=1, \ldots, m
 \end{aligned}
 \right\},
\]
and $Z [f] = Z_{\mathbb{R}} [f]$.
We say $Z[f]$ is \emph{locally finite} if $Z_I [f]$ is finite 
for all bounded open interval $I$.

To state the sufficient condition for the diffusive sign,
we also introduce a radial average of a function on $\mathbb{R}^d$
with respect to $\hat{x} \in \mathbb{R}^d$.
For a measurable function $w \colon \mathbb{R}^d \to \mathbb{R}$,
define
\begin{equation}
 \bar{w} (r;\hat{x}) 
 = 
 \left\{
 \begin{array}{ll}
  {\displaystyle \frac{w(\hat{x} + r) + w (\hat{x} - r)}{2}}
  & \mbox{if} \ d=1, \\ 
  {\displaystyle - \!\!\!\!\!\! \int_{|\omega|=1} w (\hat{x} + |r|\omega) 
   \textup{d} \mathcal{H}^{d-1} (\omega)}
   & \mbox{if} \ d \ge 2.
 \end{array}
 \right.
 \label{radial ave}
\end{equation}
The following Lemma gives a sufficient condition so that
$S_D [u_0-v_0] (\hat{x})$ is well-defined when $u_0 (\hat{x}) = v_0 (\hat{x})$.

\begin{lem}
 \label{suff condi}
 Let $u_0, v_0 \colon \mathbb{R}^d \to \mathbb{R}$ be nonnegative,
 bounded and measurable functions.
 Let $\hat{x} \in \mathbb{R}^d$ be such that
 $u_0$ and $v_0$ are continuous at $\hat{x}$, 
 and $u_0 (\hat{x}) = v_0 (\hat{x})$. Moreover, assume that
 \begin{enumerate}  
  \renewcommand{\theenumi}{A\arabic{enumi}}
  \item \label{num of disconti. of init}
	$\bar{u}_0 (r; \hat{x}), \bar{v}_0 (r; \hat{x})$ are piecewise continuous 
	(with possibly countably many discontinuities
	having at most finitely many accumulation points),
  \item \label{left-right conti. of init}
	$\bar{u}_0 (r; \hat{x}), \bar{v}_0 (r; \hat{x})$
	are either left or right continuous
	at their discontinuities.
 \end{enumerate}
 Set 
 \begin{align}
  \bar{w}_0 (r; \hat{x}) 
  & = \bar{u}_0 (r; \hat{x}) - \bar{v}_0 (kr; \hat{x}) 
  \label{radial average w} 
 \end{align}
 with $k = \sqrt{b/a}$.
 If the numbers of changes of sign 
 $Z[\bar{w}_0]$ is locally finite,
 then $S_D [u_0 - v_0] (\hat{x})$ is well-defined.
 In other words, there exists $\hat{t} > 0$ such that
 $u(\hat{x},t) - v(\hat{x},t)$ has the same sign
 for $0 < t < \hat{t}$.
 Moreover, $S_D [u_0 - v_0] (\hat{x}) \equiv 0$ if and only if
 $\bar{w}_0 (\cdot; \hat{x}) \equiv 0$.
 Thus, the set $\{ x \in \mathbb{R}^d \mid S_D [u_0 - v_0] (x) = 0 \}$
 is discrete when $d=1$, and is included in an analytic variety 
 when $d \ge 2$.
\end{lem}

If $a=b$, then our result reduces to \cite[Theorems 2.1, 2.4]{GGOU:2013CPAA}
with rescaling parameters.
Although our proof is similar to that of \cite{GGOU:2013CPAA},
we give a proof for the completeness
and for the generalization of diffusion coefficients $a$, $b$.


\begin{proof}
 In the following, we omit the notation of $\hat{x}$
 in $\bar{w}_0$, $\bar{u}_0$ and $\bar{v}_0$ for the simplicity.

 \medskip
 \noindent
 \textbf{1.}
 Set 
 \[
  \bar{w} (x,t) = \bar{u} (|x|,t) - \bar{v} (k|x|, t)
 \]
 with $k = \sqrt{b/a}$. Then, 
 $\bar{w}$ is a radial solution to 
 \begin{align}
  \label{spherecal eq-w}
  & \partial_t \bar{w} = a \Delta \bar{w}
  & & \mbox{in} \ \mathbb{R}^d \times (0,T), \\
  \label{spherecal init-w}
  & \bar{w} (\cdot,0) = \bar{w}_0 
  & & \mbox{on} \ \mathbb{R}^d.
 \end{align}
 Since the case $d=1$ and \eqref{spherecal init-w} are clear, 
 we now verify \eqref{spherecal eq-w} when $d \ge 2$. 
 We calculate $\partial_r \bar{u}$ and obtain
 \begin{align*}
  \partial_r \bar{u} (r,t) 
  & = 
  - \!\!\!\!\!\! \int_{|\omega|=1} 
  \nabla u (\hat{x} + r \omega, t) \cdot \omega 
  \textup{d} \mathcal{H}^{d-1} (\omega) \\
  & = 
  - \!\!\!\!\!\! \int_{|y-\hat{x}|=r} 
  \frac{\partial u}{\partial \nu} (y, t) 
  \textup{d} \mathcal{H}^{d-1} (\omega) \\
  & = 
  \frac{r}{d}
  - \!\!\!\!\!\! \int_{|y-\hat{x}| \le r} 
  \Delta u (y, t) \textup{d} y \\
  & = 
  \frac{1}{r^{d-1} d \alpha (d)}
  \int_{|y-\hat{x}| \le r} 
  \frac{u_t (y, t)}{a} \textup{d} y
 \end{align*}
 by Green's formula, where $\nu$ is the
 outer unit normal vector field of $\partial \{ y \in \mathbb{R}^d \mid |y-\hat{x}| < r \}$,
 and $\alpha (d)$ is the surface area of the unit sphere in $\mathbb{R}^d$.
 Thus, 
 \begin{align*}
  \partial_r (r^{d-1} \partial_r \bar{u})
  & = 
  \frac{1}{d \alpha (d)}
  \int_{|y-\hat{x}| = r} 
  \frac{u_t (y, t)}{a} \textup{d} \mathcal{H}^{d-1} (y) \\
  & = \frac{r^{d-1}}{a} 
  - \!\!\!\!\!\! \int_{|\omega|=1} 
  u_t (\hat{x} + r \omega, t) \textup{d} \mathcal{H}^{d-1} (\omega)
  = \frac{r^{d-1}}{a} \bar{u}_t.  
 \end{align*}
 Hence, we obtain
 \begin{equation*}
  \bar{u}_t = a ( \bar{u}_{rr} + \frac{d-1}{r} \bar{u}_{r} )
  = a \Delta \bar{u}  
 \end{equation*}
 since $\bar{u}$ is radially symmetric.
 Similarly, for $\bar{v}_k (r,t) := \bar{v} (kr, t)$, we have
 \[
  \partial_r \bar{v}_k (r,t) = \frac{k}{r^{d-1} d \alpha (d)} 
  \int_{|y-\hat{x}| \le kr } \frac{v_t (y, t)}{b} \textup{d}y,
 \]
 and so
 \[
  \partial_r (r^{d-1} \partial_r \bar{v}_k (r,t)) 
  = \frac{k^2}{d \alpha (d)} 
  \int_{|y-\hat{x}| = kr } \frac{v_t (y, t)}{b} 
  \textup{d} \mathcal{H}^{d-1} (y)
  = \frac{r^{d-1}}{a} \partial_t \bar{v}_k (r, t).
 \]
 Hence, we obtain
 \begin{equation*}
  \partial_t \bar{v} (kr, t) = a ( \partial_r^2 \bar{v} (kr, t) 
  + \frac{d-1}{r} \partial_r \bar{v} (kr,t))
  = a \Delta \bar{v} (kr,t).
 \end{equation*}
 The above yields \eqref{spherecal eq-w}.


 \medskip
 \noindent
 \textbf{2.}
 Note that $\lim_{r \downarrow 0} \bar{w} (r,t) = u(\hat{x},t) - v(\hat{x},t)$,
 so it suffices to investigate the sign of $\bar{w} (0,t)$.
 Since $\bar{w}_0 \equiv 0$ implies $\bar{w} \equiv 0$, 
 we now assume that $\bar{w}_0 \not\equiv 0$.

 By \eqref{num of disconti. of init}
 and \eqref{left-right conti. of init},
 $\bar{w}_0$ also satisfies them.
 By modifying the value of $\bar{w}_0$ at its discontinuities,
 we may assume that $r \mapsto \bar{w}_0 (r)$ is
 left continuous in $(0, \infty)$
 without affecting the value of $\bar{w}$ for $t > 0$.
 In summary, we now have
 \begin{enumerate}
  \renewcommand{\theenumi}{$w$\arabic{enumi}}  
  \item \label{w_0: density of discontinuity}
	$\bar{w}_0$ is piecewise continuous 
	(with possibly countably many discontinuities
	having at most finitely accumlation points),
  \item \label{w_0: left continuity}
	$\bar{w}_0$ is left continuous in $(0,\infty)$,
  \item $\bar{w}_0$ is continuous at $r=0$ and $\bar{w}_0 (0) = 0$,
  \item \label{w_0: sign changing number}
	$Z[\bar{w}_0]$ is locally finite.
 \end{enumerate}
 Immediately, 
 there exists $\gamma_0 > 0$
 such that
 \begin{enumerate}
  \item \label{w_0, continuity on edge}
	$\bar{w}_0$ is continuous in a neighborhood
	of $r = \gamma_0$,
  \item \label{w_0, nonnegativity}
	$\bar{w}_0 \ge 0$ on $[0, \gamma_0]$
	and $\bar{w}_0 (\gamma_0) > 0$, 
	or 
	$\bar{w}_0 \le 0$ on $[0, \gamma_0]$
	and $\bar{w}_0 (\gamma_0) < 0$.
 \end{enumerate}
 We first observe that
 there exists $r_k > 0$ such that \eqref{w_0, nonnegativity} of the above
 holds on $[0,r_k]$.
 In fact, 
 there exists $r_0 > 0$ such that $\bar{w}_0 (r_0) \neq 0$
 by $\bar{w}_0 \not\equiv 0$.
 We now consider the case when $\bar{w}_0 (r_0) > 0$ 
 since the other case is parallel.
 If $\bar{w_0} (r) \ge 0$ for $r \in [0,r_0]$,
 then $r_k = r_0$ is our desired one.
 Otherwise, there exists $r_1 < r_0$ such that $\bar{w}_0 (r_1) < 0$.
 If $\bar{w}_0 \le 0$ on $[0,r_1]$, then $r_k = r_1$ is our desired one.
 Otherwise, there exists $r_2 < r_1$ such that $\bar{w}_0 (r_2) > 0$.
 Inductively, there exists $r_j$ satisfying
 \begin{itemize}
  \item $0 < r_k < r_{k-1} < \cdots < r_0$, and 
  \item $\bar{w}_0 (r_{j-1}) \bar{w}_0 (r_j) < 0$
	for $j = 1, \ldots, k$.
 \end{itemize}
 Note that the set $\{ r_j \}$ is exactly 
 finite by \eqref{w_0: sign changing number}.

 We may assume that $\bar{w}_0 (r_k) > 0$.
 In this case, there exists $\delta > 0$ such that 
 $\bar{w}_0 (r) > 0$ if $r \in (r_k - \delta, r_k]$
 by \eqref{w_0: left continuity}.
 Let $D \subset (r_k - \delta, r_k]$ be a set 
 on which $\bar{w}_0$ is continuous.
 Then, there exists $\gamma_0 \in D$ such that
 $\gamma_0$ is not an accumulation point of
 discontinuities by \eqref{w_0: density of discontinuity},
 which is our desired constant.

 \medskip
 \noindent
 \textbf{3.}
 We may assume that $\bar{w}_0 \ge 0$ in $[0,\gamma_0]$ and
 $\bar{w}_0 (\gamma_0) > 0$
 since the proof of the other case is parallel.
 Under this assumption, we demonstrate that 
 there exists $\hat{t} > 0$ such that 
 $\bar{w} (0,t) > 0$ for $t \in (0,\hat{t})$,
 and thus $S_D [u_0 - v_0] (\hat{x}) = 1$.

 Recall that $\bar{w} (|x|,t)$ is a radial solution to
 \eqref{spherecal eq-w}--\eqref{spherecal init-w},
 and $\bar{w}_0 (|x|)$ is continous on $\{ |x| = \gamma_0 \}$
 by the discussion of the previous section.
 Then, we observe that $\lim_{t \downarrow 0} \bar{w} (|x|,t) = \bar{w}_0 (|x|)$
 for every $x \in \{ |x| = \gamma_0 \}$
 (see \cite[\S 2.3.1]{Evans:PDE}),
 which and compactness of $\{ |x| = \gamma_0 \}$ imply that
 there exists $\hat{t} > 0$ satisfying
 $\bar{w} (|\cdot|, t) > 0$ on $\{ |x| = \gamma_0 \}$.
 Thus, the strong maximum principle \cite{Protter-Weinberger:1967} implies that
 $\bar{w} (0,t) > 0$ for $t \in (0,\hat{t})$. 
\end{proof}

We next characterize the sets 
\[
 S_u = \{ x \in \mathbb{R}^d \mid S_D [u_0 - v_0] (x) = 1 \},
 \quad 
 S_v = \{ x \in \mathbb{R}^d \mid S_D [u_0 - v_0] (x) = -1 \}.
\]
We now introduce an essential distance between 
a point $x \in \mathbb{R}^d$ to a set $A \subset \mathbb{R}^d$
as in \cite{GGOU:2013CPAA}.
Define
\[
 d_e (x,A) = \sup \{ r \ge 0 \mid |B_r (x) \cap A| = 0 \},
\]
where $B_r (x) = \{ y \in \mathbb{R}^d \mid |y-x| < r \}$.

\begin{lem}
 \label{characterization}
 Let $A, B \subset \mathbb{R}^d$ be disjoint measurable sets,
 or more weakly $|A \cap B| = 0$.
 Let $u_0, v_0 \in L^\infty (\mathbb{R}^d)$ be nonnegative functions 
 satisfying 
 \begin{enumerate}
  \item $\mathrm{ess} \inf_A u_0 > 0$, 
	and $\mathrm{ess} \inf_B v_0 > 0$,
  \item $u_0 = 0$ on $\mathbb{R}^d \setminus A$,
	and $v_0 = 0$ on $\mathbb{R}^d \setminus B$.
 \end{enumerate}
 Then, 
 \begin{align*}
  \mathrm{int} S_u & \supset \{ x \in \mathbb{R}^d \mid d_e (x,A) < {\frac{1}{k}} d_e (x, B) \}, \\
  \mathrm{int} S_v & \supset \{ x \in \mathbb{R}^d \mid {\frac{1}{k}} d_e (x,B) < d_e (x, A) \},  
 \end{align*}
 where $k = \sqrt{b/a}$.
\end{lem}

\begin{proof}
 Since $d_e$ is continuous, it suffices to prove that $d_e (x,A) < d_e (x,B) {/k}$
 implies $x \in S_u$
 (resp. $d_e (x,A) > d_e (x,B) {/k}$ implies $x \in S_v$).
 The proof is parallel so we just give a proof of the case when
 $d_e (x,A) < d_e (x,B) {/k}$.
 {Fix $\rho > 0$ as $d_e (x,A) < \rho < d_e (x,B) / k$, then}
 \[
  |B_\rho (x) \cap A| = M > 0, 
  \quad  |B_{k \rho} (x) \cap B| = 0.
 \]
 We now represent the solutions $u$ and $v$ by the fundamental solution, i.e.,
 \begin{align*}
  u(x,t) & = \frac{1}{(4 \pi t)^{\frac{d}{2}}} 
  \int_{\mathbb{R}^d} u_0 (x - \omega) 
  \exp 
  \left( - \frac{|\omega|^2}{4at} \right) \textup{d} \omega, \\
  v(x,t) & = \frac{1}{(4 \pi t)^{\frac{d}{2}}} 
  \int_{\mathbb{R}^d} v_0 (x - \omega) 
  \exp 
  \left( - \frac{|\omega|^2}{4bt} \right) \textup{d} \omega \\
  & = \frac{k^d}{(4 \pi t)^{\frac{d}{2}}} 
  \int_{\mathbb{R}^d} v_0 (x - k \omega) 
  \exp 
  \left( - \frac{|\omega|^2}{4at} \right) \textup{d} \omega,
 \end{align*}
 and thus we have
 \[
  u(x,t) - v(x,t)
  = \frac{1}{(4 \pi t)^\frac{d}{2}}
  \int_{\mathbb{R}^d}
  (u_0 (x - \omega) - k^d v_0 (x - k \omega)) 
  \exp \left( - \frac{|\omega|^2}{4 a t} \right) \textup{d} \omega.
 \]
 We now divide this integral into two parts:
 \begin{align*}
  & u(x,t) - v(x,t)
   = \frac{1}{(4 \pi t)^\frac{d}{2}}
  (I_1 + I_2), \\
  & 
  \left\{
  \begin{aligned}
   I_1 & = 
  \int_{B_\rho (0)}
  (u_0 (x - \omega) - k^d v_0 (x - k \omega)) 
  \exp \left( - \frac{|\omega|^2}{4 a t} \right) \textup{d} \omega, \\
   I_2 & = 
   \int_{\mathbb{R}^d \setminus B_\rho (0)}
   (u_0 (x - \omega) - k^d v_0 (x - k \omega)) 
   \exp \left( - \frac{|\omega|^2}{4 a t} \right) \textup{d} \omega.
  \end{aligned}
  \right.
 \end{align*}
 On the first part, {we observe that} 
 $v_0 (x - k \omega) = 0$ when $\omega \in B_\rho (0)$,
 {which yields}
 \begin{align*}
  I_1 = \int_{B_\rho (0)} u_0 (x - \omega) 
  \exp \left( - \frac{|\omega|^2}{4 a t} \right) \textup{d} \omega
  & = \int_{B_\rho (x)} u_0 (y) 
  \exp \left( - \frac{|x-y|^2}{4 a t} \right) \textup{d} y \\
  & \ge 
  \int_{B_\rho (x) \cap A} u_0 (y) 
  \exp \left( - \frac{|x-y|^2}{4 a t} \right) \textup{d} y \\
  & \ge 
  c_A M
  \exp \left( - \frac{\rho^2}{4 a t} \right),
 \end{align*}
 where $c_A = \mathrm{ess} \inf_A u_0 > 0$.
 On the second part we have
 \begin{align*}
  I_2 
  & \ge - k^d C_B \int_{\mathbb{R}^d \setminus B_\rho (0)} 
  \exp \left( - \frac{|\omega|^2}{4 a t} \right) \textup{d} \omega \\
  & \ge - k^d \alpha (d) C_B \int_\rho^\infty 
  \exp \left( - \frac{r^2}{4 a t} \right) r^{d-1} \textup{d}r,
 \end{align*} 
 where $C_B = \mathrm{ess} \sup_B v_0$ and $\alpha (d)$ is the 
 surface area of the unit sphere in $\mathbb{R}^d$.
 Hence, we obtain
 \begin{align*}
  & u(x,t) - v(x,t) \\
  & \ge \frac{1}{(4 \pi t)^\frac{d}{2}}  
  \exp \left( - \frac{\rho^2}{4 a t} \right)
  \left\{ c_A M - k^d \alpha (d) C_B \int_\rho^\infty 
  r^{d-1} \exp \left( \frac{\rho^2 - r^2}{4at} \right) \textup{d}r \right\}.
 \end{align*}
 Since 
 \[
  0 \le 
  r^{d-1} \exp \left( \frac{\rho^2 - r^2}{4at} \right)
  \le 
  r^{d-1} \exp (\rho^2 - r^2)
  \quad \mbox{for} \ r \ge \rho \ \mbox{and} \ t \le \frac{1}{4a},
 \]
 and the righthand side of the above is integrable in $(\rho, \infty)$,
 Lebesgue's dominated convergence theorem implies that
 \[
  \lim_{t \to 0} \int_\rho^\infty 
  r^{d-1} \exp \left( \frac{\rho^2 - r^2}{4at} \right) \textup{d}r = 0.
 \]
 Hence, we obtain $u(x,t) - v(x,t) > 0$ for sufficiently small $t > 0$.
\end{proof}

As a conclusion of this section,
we obtain the following result 
from Lemmas \ref{suff condi} and \ref{characterization}.

\begin{thm}
 \label{thm: multiple class}
 Let $A_i \subset \mathbb{R}^d$ ($i=1, 2, \ldots, N$) be measurable sets 
 satisfying 
 $|A_i \cap A_j| = 0$ when $i \neq j$.
 Let $u_i (x,t)$ be a solution to 
 \begin{equation}
  \label{system of heat equations}
  \begin{array}{lcl}
   \partial_t u_i = a_i \Delta u_i & \mbox{in} & \mathbb{R}^d \times (0,T), \\
   u_i (\cdot, 0) = u_{0,i} := c_i \chi_{A_i}
    & \mbox{on} & \mathbb{R}^d
  \end{array}  
 \end{equation}
 for $i=1, 2, \ldots, N$,
 where $c_i$ and $a_i$ are positive constants.
 \begin{enumerate}
  \item \label{N-class suff condi}
	Let $\hat{x} \in \mathbb{R}^d \setminus (\bigcup_{j=1}^N \overline{A_i})$.
	Assume that $u_{0,i}$ satisfies \eqref{num of disconti. of init},
	\eqref{left-right conti. of init} for $i=1, \ldots, N$,
	and the numbers of changes of sign $Z[\bar{w}_{0,i,j}]$
	for
	\begin{equation}
	 \label{rescaling init}
	 \bar{w}_{0,i,j} (r;\hat{x})
	 := \bar{u}_{0,i} (r; \hat{x}) - \bar{u}_{0,j} (k_{i,j} r; \hat{x}),
	 \quad k_{i,j} = \sqrt{a_j/a_i}	 
	\end{equation}
	is locally finite whenever $i \neq j$.
	Then, 
	$S_D [u_{0,i} - u_{0,j}] (\hat{x})$
	is well-defined whenever $i \neq j$.
	Moreover, the relation
	$u_{0,i} \succsim u_{0,j}$, which is defined by
	$S_D [u_{0,i} - u_{0,j}] \ge 0$, is an preorder 
	of $\{ u_{0,1}, \ldots, u_{0,N} \}$ at $\hat{x}$.
  \item Under the hypothesis of \eqref{N-class suff condi},
	the classification sets defined by
	\begin{equation}
	 \label{S_i for continuum data}
	 S_i = \{ x \in \mathbb{R}^d \mid 
	 S_D [u_{0,i} - u_{0,j}] (x) = 1 
	 \quad \mbox{for} \ j \neq i \}	 
	\end{equation}
	are characterized as
	\[
	 \mathrm{int} S_i \supset \{ x \in \mathbb{R}^d \mid 
	 d_e (x,A_i) < {\frac{1}{k_{i,j}}} d_e (x, A_j) 
	 \quad \mbox{for} \ j \neq i \}.
	\]
 \end{enumerate}
\end{thm}

\subsection{For discrete data}
\label{sec: for discrete data}

At the practical level, each training data $A_i$ will be given as 
discrete data: 
\begin{equation}
 \label{discrete training data}
 A_i = \{ p_{i,j} \in \mathbb{R}^d \mid j =1, 2, \ldots, K_i \}
 \quad \mbox{for} \ j = 1, 2, \ldots, N  
\end{equation}
for constants $K_i \in \mathbb{N}$.
In such a case, the initial data $u_{0,i}$ 
defined by \eqref{initial data}
satisfies $u_{0,i} \equiv 0$ in $L^\infty (\mathbb{R}^d)$,
which implies $u_i \equiv 0$ for $i=1, \ldots, N$, 
and thus the diffusive method mathematically
fails to produce meaningful results.
This issue was addressed in the previous work \cite{GGOU:2013CPAA},
where each point data $p_{i,j}$ was expanded into a small cube $Q_{i,j}$, 
and the initial data $u_{0,i}$ was defined as \eqref{initial data}
with $A_i = \bigcup_{j=1}^{K_i} Q_{i,j}$.
While this regularization procedure works well in the case 
$d=2$, in the present paper 
we propose a different approach that directly employs 
the heat kernel to scale the algorithm to higher dimensions.

Our proposed approach is to use the heat kernel directly.
For given data $A_i$ and diffusion constant $a_i > 0$, 
define 
\begin{equation}
 \label{def Psi}
 \Psi_i (x,t) = \sum_{j=1}^{K_i}
 \frac{1}{(4 \pi t)^{\frac{d}{2}}}
 \exp \left( - \frac{|x-p_{i,j}|^2}{4 a_i t} \right)
 \quad 
 \mbox{for} \ i=1, \ldots, N 
\end{equation}
and set
\begin{equation} \label{S_i}
 S_i = \{ x \in \mathbb{R}^d \mid 
 \lim_{t \to 0} \mathrm{sgn} (\Psi_i (x,t) - \Psi_j (x,t)) = 1 
 \quad \mbox{for every} \ j \neq i \}.
\end{equation} 

\begin{thm}
 For $i=1, 2, \ldots, N$, the set $S_i$ is characterized as 
 \begin{align*}
  \mathrm{int} S_i & \supset P_i := \{ x \in \mathbb{R}^d \mid 
  \mathrm{dist} (x, A_i) < \mathrm{dist} (x, A_j)/k_{i,j} 
  \quad \mbox{for every} \ j \neq i \}, \\
  \mathrm{int} S_i^c & \supset \{ x \in \mathbb{R}^d \mid 
  \mbox{there exists $j_0 \neq i$ such that} \
  \mathrm{dist} (x, A_i) > \mathrm{dist} (x, A_{j_0})/k_{i,j_0} \},  
 \end{align*}
 where $k_{i,j} = \sqrt{a_j / a_i}$
 and $\mathrm{dist} (x, A_i) = \min \{|x-y| \mid y \in A_i \}$.
 In particular, if $a_i = a_1$ for $1 \le i \le N$,
 then the set $S_i$ is characterized as $\mathrm{int} S_i = P_i$.
 In other words, $S_i$ gives the nearest neighbor classification for 
 $\{ A_i \mid i=1, \ldots, N \}$.
\end{thm}

\begin{proof}
 We now assume that 
 $x \in P_i$ 
 to show $x \in \mathrm{int} S_i$.
 Let us set
 \begin{align*}
  f_{i,j} (y) := \mathrm{dist} (y, A_i)^2 - \frac{\mathrm{dist} (y,A_j)^2}{k_{i,j}^2}.
 \end{align*}
 Since $y \mapsto \mathrm{dist} (y,A_j)$ is continuous for every $j = 1, \ldots, N$,
 there exists $\delta_{i,j} > 0$ such that
 \begin{align*}
  f_{i,j} (y) < \frac{f_{i,j} (x)}{2} < 0
  \quad \mbox{for} \ j \neq i \ 
  \mbox{whenever} \ |y - x| < \delta_{i,j}.  
 \end{align*}
 This implies that
 \begin{align*}
  & \Psi_i (y,t) - \Psi_j (y,t) \\
  & \ge \frac{1}{(4 \pi t)^{\frac{d}{2}}}
  \exp \left( - \frac{\mathrm{dist} (y, A_i)^2}{4 a_i t} \right)
  - \frac{K_j}{(4 \pi t)^{\frac{d}{2}}}
  \exp \left( - \frac{\mathrm{dist} (y, A_j)^2}{4 a_j t} \right) \\
  & = \frac{1}{(4 \pi t)^{\frac{d}{2}}}
  \exp \left( - \frac{\mathrm{dist} (y, A_i)^2}{4 a_i t} \right)
  \left\{
  1 - K_j
  \exp \left( \frac{f_{i,j} (y)}{4 a_i t} 
  \right)
  \right\} \\
  & \ge \frac{1}{(4 \pi t)^{\frac{d}{2}}}
  \exp \left( - \frac{\mathrm{dist} (y, A_i)^2}{4 a_i t} \right)
  \left\{
  1 - K_j
  \exp \left( \frac{f_{i,j} (x)}{8 a_i t} 
  \right)
  \right\}.
 \end{align*}
 Recall that $f_{i,j} (x) < 0$, and thus there exists $\hat{t}_{i,j} > 0$
 such that
 \[
  \Psi_i (y,t) - \Psi_j (y,t) > 0 
  \quad \mbox{for} \ 0 < t 
  < \hat{t}_{i,j}
  \ \mbox{and} \ |y-x| < \delta_{i,j}
 \]
 if $j \neq i$,
 which implies $x \in \mathrm{int} S_i$.

 On the other hand, assume that 
 there exists $j_0 \neq i$ such that
 \[
  \mathrm{dist} (x, A_i) > \mathrm{dist} (x. A_{j_0}).
 \]
 By the continuity of $\mathrm{dist} (x, A_j)$
 for $j=1, \ldots, N$, 
 there exists $\delta > 0$ such that
 \[
  f_{i,j_0} (y) > \frac{f_{i,j_0} (x)}{2} > 0
  \quad \mbox{if} \ |y-x| < \delta.
 \]
 Fix $y$ satisfying $|y - x| < \delta$.
 We now observe that
 \begin{align*}
  & \Psi_i (y,t) - \Psi_{j_0} (y,t) 
  \le \frac{1}{(4 \pi t)^{\frac{d}{2}}}
  \exp \left( - \frac{\mathrm{dist} (y,A_i)^2}{4 a_i t} \right)
  \left\{ K_i 
  - \exp \left( \frac{f_{i,j_0} (x)}{8 a_i t} \right) \right\}.
 \end{align*}
 The above implies that there exists $t^* > 0$ such that
 \[
  \Psi_i (y,t) - \Psi_{j_0} (y,t)
  < 0 \quad \mbox{for} \ 0 < t < t^*,
 \]
 and thus $\lim_{t \to 0} \mathrm{sgn} (\Psi_i (y,t) - \Psi (y,t)) = -1$
 for $|y - x| < \delta$, that is, $x \in \mathrm{int} S_i^c$.

 Finally, assume that $a_i = a_1$ for $i=1, \ldots, N$ and thus $k_{i,j} = 1$.
 We now demonstrate that $P_i^c \subset (\mathrm{int} S_i)^c$.
 Let $x \in P_i^c$.
 There exists $j_0 \neq i$ such that $f_{i,j_0} (x) \ge 0$.
 If $f_{i, j_0} (x) > 0$, then we obtain $x \in \mathrm{int} S_i^c 
 \subset (\mathrm{int} S_i)^c$ by the previous argument.
 Thus, it remains the case when $f_{i, j_0} (x) = 0$,
 that is, $\mathrm{dist} (x, A_i) = \mathrm{dist} (x, A_{j_0})$.
 Let $y \in A_i$ and $z \in A_{j_0}$ be such that
 \[
  |x-y| = \mathrm{dist} (x, A_i), \quad
  |x-z| = \mathrm{dist} (x, A_{j_0}).
 \]
 Note that $y \neq z$, but $|x-y| = |x-z|$.

 Set
 \[
  \omega_\lambda = (1-\lambda) x + \lambda y 
  \quad \mbox{for} \ \lambda \in (0,1).
 \]
 We claim that
 \begin{equation}
  \mathrm{dist} (\omega_\lambda, A_i) < \mathrm{dist} (\omega_\lambda, A_{j_0})
  \quad \mbox{for every} \ \lambda \in (0,1),
  \label{claim:2.15}
 \end{equation}
 and thus $f_{i,j_0} (\omega_\lambda) < 0$ for $\lambda \in (0,1)$.
 This yields 
 $\lim_{t \to 0} 
 \mathrm{sgn} (\Psi_i (\omega_\lambda, t) - \Psi_{j_0} (\omega_\lambda, t)) = -1$
 for $\lambda \in (0,1)$.
 Hence, we obtain $x \notin (\mathrm{int} S_i)^c$.

 We now prove \eqref{claim:2.15}.
 Let $z_0$ be such that 
 $|\omega_\lambda - z_0| = \mathrm{dist} (\omega_\lambda, A_{j_0})$.
 First, since $|\omega_\lambda - y| = (1-\lambda)|x-y|$
 we obtain
 \begin{equation}
  \mathrm{dist} (\omega_\lambda, A_i) \le (1-\lambda)|x-y|
  \quad \mbox{for} \ \lambda \in (0,1).
  \label{claim:2.16}  
 \end{equation}
 On the other hand, 
 \begin{equation}
  \mathrm{dist} (\omega_\lambda, A_{j_0}) \ge (1-\lambda)|x-z|
   \quad \mbox{for} \ \lambda \in (0,1).
   \label{claim:2.17}
 \end{equation}
 In fact, 
 if $|\omega_\lambda - z_0| < (1 - \lambda) |x-z|$, then
 \[
  |x-z_0| \le |x - \omega_\lambda| + |\omega_\lambda - z_0| 
  < \lambda |x-y| + (1 - \lambda) |x-z| = |x-z|,
 \]
 which contradicts the definition of $z$.
 Thus, we obtain
 \[
  \mathrm{dist} (\omega_\lambda, A_i) \le \mathrm{dist} (\omega_\lambda, A_{j_0}).
 \]
 Now, if 
 $\mathrm{dist} (\omega_\lambda, A_i) = \mathrm{dist} (\omega_\lambda, A_{j_0})$,
 then \eqref{claim:2.16} and \eqref{claim:2.17} imply that
 $|\omega_\lambda - z_0| = (1 - \lambda)|x-z|$.
 Furthermore, 
 \[
  |x-z_0| \le |x - \omega_\lambda| + |\omega_\lambda - z_0| 
  = \lambda |x-y| + (1-\lambda)|x-z|
  \le \mathrm{dist} (x, A_{j_0}),
 \]
 and thus $|x - z_0| = \mathrm{dist} (x, A_{j_0}) = |x-z|$.
 Therefore, equality holds throughout the above of inequalities. 
 By the equality case of the triangle inequality,
 $x - \omega_\lambda$ and $\omega_\lambda - z_0$ are parallel.
 Hence, $z_0 = y$ and thus $z_0 \in A_i \cap A_{j_0}$,
 a contradiction.
 This proves \eqref{claim:2.15}.
\end{proof}

{

\section{Algorithm}
\label{sec: algorithm}

\subsection{Fundamental approach}

A fundamental algorithm of the classification 
by \eqref{S_i for continuum data}
for continuum training data $A_i$, $i=1, \ldots, N$ 
is just solving the heat equation \eqref{system of heat equations}
in a very short time.
Namely, the classification at practical level is 
obtained by
\[
 S_i^h = \{ x \in \mathbb{R}^d \mid u_i (x,h) > u_j (x,h) 
 \quad \mbox{for} \ j \neq i \}
\]
for a sufficiently small constant $0 < h \ll 1$.
Note that, numerically, 
$u_i (x, h) \ll 1$ where $\chi_{A_i} (x) = 0$ since 
$\lim_{t \to 0} u_i (x,t) = 0$.
Therefore, when we solve \eqref{system of heat equations}
with a explicit finite difference scheme, 
we need an enough numerical time step to obtain 
accurate $S_i^h$.

When we have the discrete training data \eqref{discrete training data},
a fundamental algorithm of the classification 
is using a single parameter $h$ derived from the training dataset for all inputs.
To find the single parameter $h$, 
we first compute $\Psi_k (p_{i,j}, h)$ defined by
\eqref{def Psi}, 
and count the incorrect predictions of training data,
which is defined by
\begin{equation}
 \label{fault number}
  \Lambda_h = \# \left\{ 
 (i,j) \in \bigcup_{\ell = 1}^N \{ \ell \} \times \{ 1, \ldots K_\ell \} 
 \ \middle| \ 
 \begin{aligned}
  & \mbox{there exists} \ k \neq i \ \mbox{such that} \\
  & \Psi_k (p_{i,j}, h) > \Psi_i (p_{i,j}, h)
 \end{aligned}  
 \right\} 
\end{equation}
for $h > 0$.
Note that $\Lambda_h$ may be positive for sufficiently
large $h$. 
On the other hand, clearly there exists $h_* > 0$ such that
$\Lambda_h = 0$ for $h \in (0,h_*)$, 
since 
\begin{align*}
 & \lim_{h \to 0} \Psi_i (p_{i,j},h) = \infty
 \quad \mbox{for every} \ 
 j=1, \ldots, K_i, \ i=1, \ldots, N, \\ 
 & \lim_{h \to 0} \Psi_k (p_{i,j}, h) = 0
 \quad \mbox{for} \ k \neq i.
\end{align*}
We now find $h_0 > 0$ such that $\Lambda_{h_0} = 0$
by successively replacing $h$ with $r h$ while $\Lambda_h > 0$,
where $r \in (0,1)$ is a constant (e.g., $r=1/2$).
This procedure yields the desired parameter $h_0 > 0$,
that is $\Lambda_{h_0} = 0$.
Hence, we obtain the baseline classifier using 
$\Psi_{i,j} (x, h_0)$; 
set
\[
 S_i^{h_0} = \{ x \in \mathbb{R}^d \mid \ 
 \Psi_i (x, h_0) = \max_{1 \le k \le N} \Psi_k (x, h_0) \}.
\]
This algorithm is summarized in Algorithm~\ref{algo: fundamental}.

\begin{algorithm}[H] 
    \caption{Baseline approach for discrete training data}
    \label{algo: fundamental}
    \begin{algorithmic}  
    \STATE Data: Training data 
     $A_i = \{ p_{i,j} \mid j = 1, 2, \ldots, K_i \}$ for $i = 1, 2, \ldots, N$
    \STATE Define $\Psi_i (x, t)$ by \eqref{def Psi} for $i = 1, 2, \ldots, N$
    \STATE Define $r \in (0,1)$
    \STATE Initialize $h \leftarrow 1$
    \STATE Define $\Lambda_h$ by \eqref{fault number}
    \WHILE{$\Lambda_h > 0$}
    \STATE $\displaystyle h \leftarrow r h$
    \ENDWHILE
    \STATE Find $i$ such that 
     $\Psi_i(x, h) = \max_{1 \le k \le N} \Psi_k (x, h)$ 
     for every testing data $x$
    \end{algorithmic}
\end{algorithm}

Let us denote the predicted label of an input $x \in \mathbb{R}^d$
by $i(x) \in \{1, \ldots, N \}$.
Clearly, Algorithm~\ref{algo: fundamental} yields that
$i(p_{\ell, j}) = \ell$
for all training data $p_{\ell,j} \in A_\ell$
(for every $\ell = 1, \ldots, N$ and $j = 1, \ldots, K_\ell$).
Our method follows 
the empirical risk minimization in this sense.
This algorithm works well for lower spatial dimension 
and a small range of the data,
like as $d=2$ and $p_{i,j} \in [-1,1]^d$.

\subsection{Algorithms for higher dimensional data}
\label{sec: 3.2}

At the practical level of pattern recognition,
the dimension $d$ of input space is too high.
In fact, for example, the testing and training data
of MNIST database \cite{MNIST_Url,MNIST}
are provided as $x \in [0,255]^d$ 
vector data with $d=28 \times 28$ dimension,
which we use in next section.
Therefore, 
numerical underflow or overflow often occurs in the
computation of $\Psi_i$ 
because of the extremely huge quantity
$(4 \pi t)^{-d/2}$ and extremely small number $\exp (- |x-p_{i,j}|^2 / (4t))$
for $t \ll 1$,
even if we normalize that, for example,  $x \in [0,1]^d$.
Thus, we now propose an algorithm to establish 
the idea of diffusive sign for such a high dimensional
feature space.

From now on, we put $a_i = 1$ for every $i = 1, 2, \dots, N$ for simplicity. 
For an input $x \in \Omega$,
the problem is to find a label $i = 1, 2, \ldots, N$ such that 
$x \in S_i$, where $S_i$ is the set given by  \eqref{S_i}. 
Clearly, by  \eqref{S_i} $x \in S_i$ if and only if 
there exists $t^* > 0$ such that
\[
 \Psi_i(x, t) > \Psi_j(x,t)
 \quad \mbox{for} \ j \neq i
 \ \mbox{provided that} \ t \in (0,t^*),
\]
that is, 
\[
 \Psi_i(x, t) = \max \{ \Psi_j(x,t) \mid \ j = 1, 2, \dots, N \}
 =: \widehat{\Psi} (x,t)
 \quad \mbox{for} \ t \in (0,t^*).
\] 
To address the issue of underflow or overflow in 
the computation of $\Psi_i (x,t) - \Psi_j (x,t)$, 
we normalize it to
\[
 (4 \pi t)^{\frac{d}{2}} (\Psi_i (x,t) - \Psi_j (x,t)) 
 = 
 \sum_{k=1}^{K_i} \exp \left( - \frac{|x - p_{k,i}|^2 }{4 t} \right)
 - \sum_{k=1}^{K_j} \exp \left( - \frac{|x - p_{k,j}|^2 }{4 t} \right) 
\]
which limits the numerical issues to 
underflow only.
Thanks to the normalization, 
our classification algorithm is now rephrased with
the functions 
\begin{equation}
 \Phi_i (x, \alpha) = \sum_{j=1}^{K_i}
 \exp \left( - \alpha  |x-p_{i,j}|^2  \right)
 \mbox{ for } \ i=1, \ldots, N, \ 
 x \in \mathbb{R}^d  \mbox{ and }  \alpha > 0,  \label{def_Phi}
\end{equation}
that is, find the label $i (x) \in \{ 1, 2, \ldots, N \}$
such that
there exists $\alpha^* > 0$ satisfying
\begin{equation} 
 \label{def_i} 
  \Phi_{i (x)} (x, \alpha) = \max \{ \Phi_j(x, \alpha) \mid  \ j = 1, 2, \dots, N \}
 =: \widehat{\Phi} (x,\alpha)
 \quad \mbox{for} 
 \ \alpha \in (0,\alpha^*) 
\end{equation}
provided that $\Phi_j (x, \alpha) \ll 1$
so that computed value of $\Phi_j (x, \alpha)$ often underflows
for every $j \in \{ 1, \ldots, N\}$.


In order to get the label $i(x)$ satisfying \eqref{def_i},
we exploit this underflow behavior as part of the computational procedure.
Specifically, for $x \notin \bigcup_{i=1}^N A_i$,
we first choose $\alpha_0 > 0$ \emph{sufficiently large} 
so that all computed values $\Phi_i (x, \alpha_0)$ fall below
machine precision and hence underflow.
In other words, for sufficienlty small $\varepsilon > 0$,
we find $\alpha_0 > 0$ satisfying
\begin{align*}
 & \widehat{\Phi} (x, \alpha_0) < \varepsilon.
\end{align*}
Next, we gradually decrease $\alpha_0$ 
to find the label $i(x)$ whose value is the first
to emerge from underflow, that is, find
$\alpha^* \le \alpha_0$ and $i(x) \in \{ 1, \ldots, N \}$
so that
\begin{align*}
 & \alpha^* = \inf \{ \hat{\alpha} \mid \widehat{\Phi} (x, \alpha) < \varepsilon
 \ \mbox{for} \ \alpha \in [\hat{\alpha}, \infty) \}
 \quad 
 \mbox{and} \quad
 \Phi_{i(x)} (x, \alpha^*)
 = \widehat{\Phi} (x, \alpha^*).
\end{align*}
At the practical level, set $r \in (0,1)$ and
$\alpha_k = r^k \alpha_0$,
and compute $\Phi_i (x, \alpha_k)$ for every 
$i = 1, \ldots, N$ to find $k_0 \ge 0$ such that
\[
 \widehat{\Phi} (x, \alpha_{k_0}) \ge \varepsilon, \quad
 \widehat{\Phi} (x, \alpha_k) < \varepsilon
 \quad \mbox{for} \ 0 \le k < k_0.
\]
Then, there exists a label $i \in \{ 1 , \ldots, N \}$
so that $\Phi_i (x, \alpha_k) \ge \varepsilon$.
Hence, we set the label $i(x)$ of $x$ such that
\[
 \Phi_{i(x)} (x, \alpha_k) = \widehat{\Phi} (x, \alpha_k).
\]
Due to this argument, we propose the following algorithm
to find $i(x)$ satisfying \eqref{def_i}.
(See also Algorithm~\ref{alg1}.)

\begin{enumerate}
 \item Set the constants $r \in (0,1)$ and $\varepsilon > 0$,
       where $\varepsilon$ is chosen so that
       the numerical result of $\Phi_j (x, \alpha)$ underflows when
       $\Phi_j (x, \alpha) \le \varepsilon$.
 \item Initialize $\alpha = \alpha_0$ satisfying $\Phi_j (x, \alpha_0) \le \varepsilon$
       for every $j=1, \ldots, N$.
 \item \label{alg2: update alpha}
       Compute $\Phi_i (x, \alpha)$ for every $i=1, \ldots, N$.
 \item Check the values of $\Phi_i (x, \alpha)$.
       \begin{enumerate}
	\item If $\Phi_i (x, \alpha) \le \varepsilon$ for every $i=1, \ldots, N$,
	      then update $\alpha$ to $r \alpha$ with a constant $r \in (0,1)$
	      and return to \eqref{alg2: update alpha}.
	\item If there exists $i \in \{1, \ldots, N \}$ such that
	      $\Phi_i (x, \alpha) > \varepsilon$
	      (i.e., there is a numerical result which does not
	      underflow), 
	      then find 
	      \[
	       i(x) = \arg \max_{1 \le i \le N} \Phi_i (x, \alpha)
	      \]
	      as the classification result of the datum $x$.
       \end{enumerate}
\end{enumerate}

\begin{algorithm}[H] 
    \caption{Find  $i$ satisfying \eqref{def_i}
 for each testing datum $x$}
    \label{alg1}
    \begin{algorithmic}  
    \STATE Data: Training data 
     $A_i = \{ p_{i,j} \mid j = 1, 2, \ldots, K_i \}$ for $i = 1, 2, \ldots, N$.
    \STATE Define $\Phi_i (x, \alpha)$ by \eqref{def_Phi} for $i = 1, 2, \ldots, N$.
    \STATE Define $r \in (0,1)$ and $\varepsilon \ge 0$.
     \STATE Initialize $\alpha \leftarrow \alpha_0$
     so that $\max \{ \Phi_j (x, \alpha) \mid \ j=1, \ldots, N \} \le \varepsilon$.
     \WHILE{$\max\{ \Phi_j(x, \alpha) \mid  j = 1, 2, \dots, N\} \le \varepsilon$}
    \STATE $\displaystyle \alpha \leftarrow r \alpha$
    \ENDWHILE
    \STATE Find $i (x)$: 
     $\Phi_{i (x)} (x, \alpha) = \max\{ \Phi_j(x, \alpha) \mid  j = 1, 2, \dots, N\}$ 
    \end{algorithmic}
\end{algorithm}



When we get some new classification results,  
it is very easy to update our
discreminants. 
Indeed, when we get a pair of a new datum $p \in {\mathbb R}^d$ 
and its classification label $i$, 
we can update our 
discreminants just by adding the term
\[
 \tilde{\Phi}_i (x, \alpha) = \sum_{j=1}^{K_i}
  \exp \left( - \alpha  |x-p_{i,j}|^2  \right) +  \exp \left( - \alpha  |x-p|^2  \right).  
\]
Algorithm~\ref{alg3} presents the online learning algorithm of
Algorithm~\ref{alg1} due to this idea. 

\begin{algorithm}[H]
 \caption{Online learning algorithm 
 }
    \label{alg3}
    \begin{algorithmic}  
     \STATE Data: Training data $A_i = \{ p_{i,j} \mid 
     j = 1, 2, \ldots, K_i \}$ for $i = 1, 2, \ldots, N$,  
     \STATE Data: Testing datum $(p_m, \ell_m)$ for $m = 1, 2, \ldots, M$ 
     \STATE Define $\Phi_i (x, \alpha)$ by \eqref{def_Phi} for $i = 1, 2, \ldots, N$
     \STATE Define $r \in (0,1)$ and $\varepsilon > 0$
     \STATE $\hat{\Phi}_i (x, \alpha) \leftarrow \Phi_i (x, \alpha) $
    \FOR  {$m = 1$ to $M$}
    \STATE Initialize $\alpha \leftarrow \alpha_0$
     so that $\hat{\Phi}_i (p_m, \alpha) \le \varepsilon$ for $i=1, \ldots, N$
    \WHILE{$\max \{ \hat{\Phi}_i (p_m, \alpha) \mid i = 1, \ldots, N \} \le \varepsilon$}
    \STATE $\displaystyle \alpha \leftarrow  r \alpha$
    \ENDWHILE
    \STATE Find $i (p_m)$: 
     $\hat{\Phi}_{i (p_m)} (p_m, \alpha) 
     = \max\{ \hat{\Phi}_j (p_m, \alpha) \mid  j = 1, 2, \dots, N \}$ 
     \STATE $\hat{\Phi}_{i (p_m)} (x, \alpha) \leftarrow \hat{\Phi}_{i (p_m)} (x, \alpha) 
     + \exp(-\alpha |x - p_m|^2)$
    \ENDFOR 
    \end{algorithmic}
\end{algorithm}
The effectiveness of this improvement is shown in \S \ref{MNIST test}.

\section{Numerical results}
\label{numerical_results} 


\subsection{Geometrical results}

As an example of numerical results by our classification algorithm,
Figure \ref{fig: result-1_triple} presents the result of classification
for 3-class data by the fundamental algorithm.
In this simulation, we prepare random 20 data for an each class of
$3$-classes in $[-2,2]^2$, which are indicated 
by the points $\blacksquare$, $\blacktriangle$ and $\bullet$ in the figure.
On other parameters, we set $a_j = 1$ for $j=1,2,3$,
and $t = 0.00390625$ in this case for breaking the loop
finding $t > 0$.
\begin{figure}[htbp]
 \begin{center}
  \includegraphics[scale=0.3]{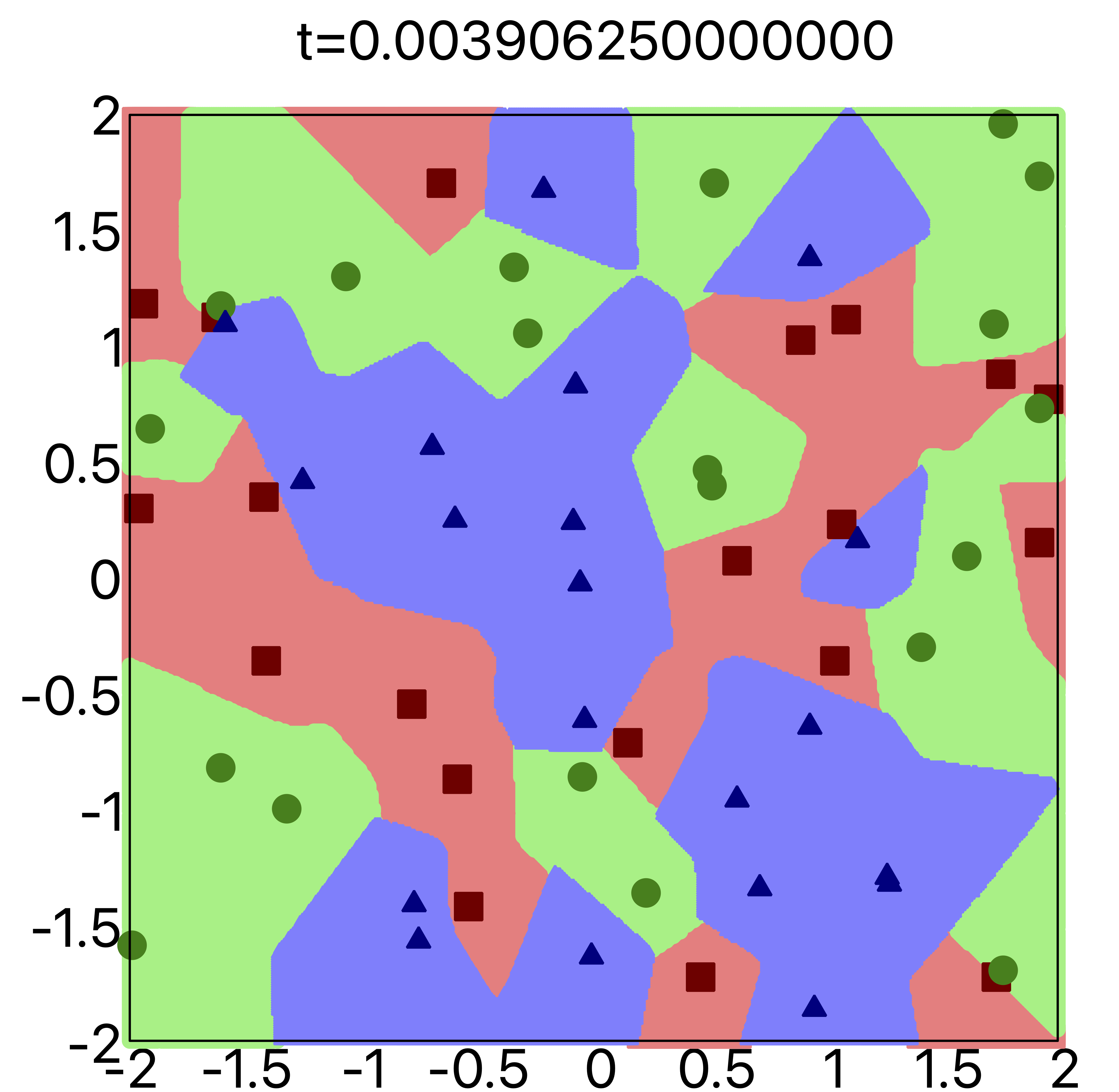}
  \caption{Numerical result of 3-class data classification.}
  \label{fig: result-1_triple}
 \end{center}
\end{figure}

The difference of diffusion coefficient yields bubble-like classification
result, since $|x-a| < k |x-b|$ yields a ball 
for every $x, a, b \in \mathbb{R}^d$ and $0 < k < 1$.
Figure \ref{fig: result-3-4_triple} indicates the classification
results 
\begin{itemize}
 \item (Left pannel) with $a_1=2$, $a_2 = a_3 = 1$,
 \item (right pannel) with $a_1 = a_2 = a_3 = 1$.
\end{itemize}

\begin{figure}[htbp]
 \begin{center}
  \begin{tabular}{c|c}
   \includegraphics[scale=0.2]{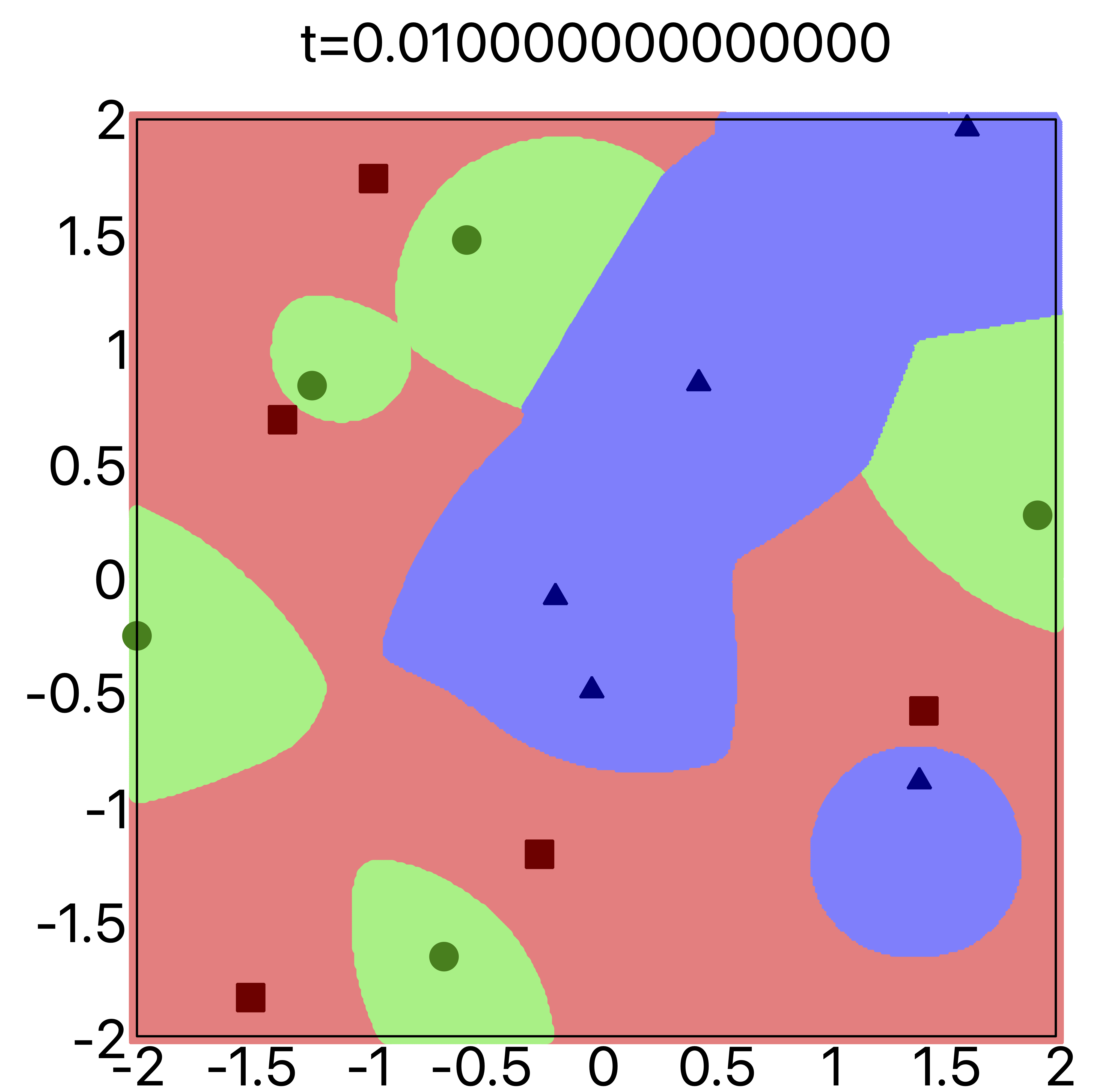} 
   & \includegraphics[scale=0.2]{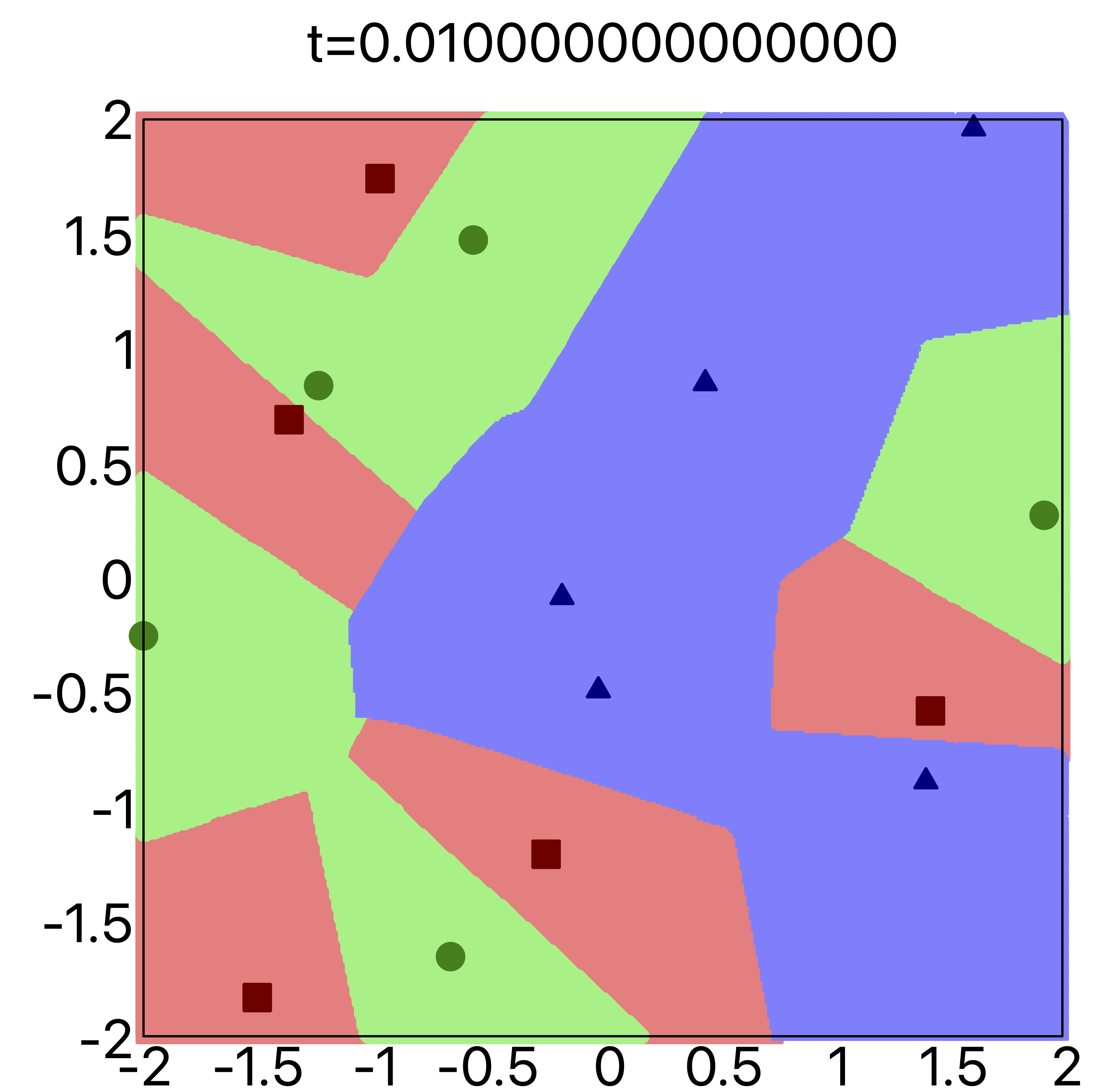} \\
   Case $a_1 = 2$, $a_2 = a_3 = 1$. &
   Case $a_1 = a_2 = a_3 = 1$.
  \end{tabular}  
  \caption{Comparison for changing diffusion coefficients.}
  \label{fig: result-3-4_triple}
 \end{center}
\end{figure}

\subsection{Numerical results for handwritten digit database}
\label{MNIST test}

In this section, we demonstrate the classification results
for handwritten digits database provided by MNIST.
It is well-known as a large database of handwritten digits 
and contains 60,000 training images and 10,000 testing images.
We here denote the number of the training data of handwritten
integer $\ell \in \{ 0, 1, \ldots, 9 \}$ by $K_\ell$,
and thus $\sum_{\ell=0}^9 K_\ell = 60,000$.  
Each image is given by gray tone, 
and the size is $d=28 \times 28$ pixel. 
Thus, we denote the training data by MNIST as
\begin{equation}
 A = \{ (\xi^k, \eta_k) \in [0,255]^d
  \times \{ 0, 1, \ldots, 9 \} \mid k=1, 2, \ldots, 60000 \}. 
  \label{mnist training data}
\end{equation}
We also use the notation
\[
 A_\ell  = \{ p^k_\ell \in [0, 255]^d
 \mid k = 1, 2, \ldots, K_\ell \}, 
\] 
for the data and the set of handwritten digits $\ell \in \{ 0, 1, \ldots, 9 \}$
when we focus on the data of the digit $\ell$.
For a datum of handwritten digits $\xi = (\xi_{i,j})_{1 \le i,j \le 28}$, 
$\xi_{i,j}=0$ (resp. $=255$)  means that the color of the 
$(i,j)$-th pixel is black (resp. white). 
Thus, as mentioned in the previous section, the classification function 
$\Phi_\ell (x, \alpha)$ corresponding to the digit $\ell$ is 
\[
\Phi_\ell (x, \alpha) = \sum_{k=1}^{K_\ell}
  \exp \left( - \alpha  |x-p^k_\ell|^2  \right). 
\]

Here and hereafter, we denote the numerical result of $\Phi_\ell (x, \alpha)$
by $\Phi_{\Delta, \ell} (x, \alpha)$, and 
$\widehat{\Phi}_{\Delta} (x, \alpha) 
= \max_{\ell \in \{ 1, \ldots, N \}} \Phi_{\Delta, \ell} (x, \alpha)$.
We mainly choose the threshold parameter $\varepsilon$
of our algorithms in \S \ref{sec: 3.2} as the machine epsilon.
We denote the machine epsilon by $\varepsilon_*$.

\par
\bigskip
\noindent
\textbf{Baseline approach.}

\newcounter{observation}

As a baseline approach of our algorithm,
we first examine our approach using a single parameter derived from
the training dataset for all inputs.
In other words, we give a pre-determined function $\Phi_\ell (x, \alpha)$ 
with a uniform constant $\alpha > 0$ as the discreminants of 
the classification.

We first find a range of $\alpha$ which enables us to use it 
for classifier.
For this purpose, we investigate the accuracy and error rate 
for $\alpha = 0.1^p$ with several integers $p$.
Note that $\Phi_\ell (p^k_\ell, \alpha) \ge 1$ for any $\ell = 0, 1, \ldots, 9$
and $p^k_\ell \in A_\ell$, thus no underflow occurs for training data.
Table~\ref{table: classification of training data} presents the 
numerical results of accuracy and error rates of the training data 
for $\alpha = 0.1^p$ with $p=0, 1, \ldots, 7$.
From the viewpoint of empirical risk minimization,
one can find that $\alpha \ge 0.1^5$ is reasonable to choose
a uniform parameter of $\alpha$ for the baseline approach.
\begin{table}[htbp]
 \begin{center}
  \renewcommand{\arraystretch}{1.2}
  \begin{tabular}{|c|c|c|}
   \hline
   $\alpha$ & Correct & Incorrect \\
   \hline
   $1.0$ & 60000 (100.00\%) & 0 (0.00\%) \\
   \hline
   $0.1$ & 60000 (100.00\%) & 0 (0.00\%) \\
   \hline
   $0.1^2$ & 60000 (100.00\%) & 0 (0.00\%) \\
   \hline
   $0.1^3$ & 60000 (100.00\%) & 0 (0.00\%) \\
   \hline
   $0.1^4$ & 60000 (100.00\%) & 0 (0.00\%) \\
   \hline
   $0.1^5$ & 60000 (100.00\%) & 0 (0.00\%) \\
   \hline
   $0.1^6$ & 47942 (79.90\%) & 12058 (20.10\%) \\
   \hline
   $0.1^7$ & 22981 (38.30\%) & 37019 (61.70\%) \\
   \hline
  \end{tabular}
  \caption{Classification results (count / rate) of training data 
  for uniformly fixed $\alpha = 0.1^p$ ($0 \le p \le 7$).
  }
  \label{table: classification of training data}
 \end{center}
\end{table}

We next investigate if the baseline approach is useful 
to classify the test data of handwritten digits provided by MNIST.
Now, let us denote the test data by MNIST as
\begin{equation}
 S = \{ (x^k, y_k) \in [0,255]^d \times \{ 0, 1, \ldots, 9 \} \mid
 k = 1, 2, \ldots, 10000 \}.
 \label{mnist test data}
\end{equation}
For these data, we investigate the accuracy and the rates of
error and no decision for the baseline approach.
Note that ``No decision'' means that $(x^k, y_k)$
results 
$\widehat{\Phi}_\Delta (x^k, \alpha) \le \varepsilon$,
that is, our method cannot answer its label.
Table~\ref{table: classification of test data} presents 
the counts and rates of Correct, Incorrect and No decision
of test data for $\alpha = 0.1^p$ with $p=0, \ldots, 7$.
One can find that 
\begin{enumerate}
 \item There are too much No decision case when $\alpha \ge 0.1^3$.
       \label{much vanish}
 \item On the other hand, there are too much Incorrect case when $\alpha \le 0.1^6$.
       It is consistent of the result for training data.
 \item In summary, $0.1^5 \le \alpha \le 0.1^4$ is reasonable 
       for MNIST database.
       In particular, our approach with $\alpha = 0.1^5$
       results high accuracy of $96.95 \%$ even though
       we do nothing to the test data 
       for prelimiary of the classification. 
       \setcounter{observation}{\value{enumi}}
\end{enumerate}
According to \cite{MNIST}, the baseline algorithm of our approach
establishes better accuracy than $k$-Nearest Neighbor ($k$-NN) classification
with Euclidean norm, nonlinear classifier 
by 40 PCA with quadratic, 
and linear classifier by 1000 radial basis functions approaches.
(See also \cite{LindbladSladoje_2014IEEE, Keysers_etal_2007IEEE}
for $k$-NN; they establishes higher accuracy than our results.)
On the other hand,
observation \eqref{much vanish} indicates that the computation of
$\Phi_{\Delta, \ell} (x^k, \alpha)$ frequently underflows when $\alpha$ is large
(equivalently, when $t$ is small).
Consequently, regardless of how $\alpha$ is chosen, 
a large number of both ''Incorrect'' and ''No decision''
outcomes may occur,
making it difficult to identify an appropriate
value of $\alpha$ for the discriminants.
This motivates the consideration of a pointwise algorithm,
such as Algorithm 2. 
\begin{table}[htbp]
 \begin{center}
 \renewcommand{\arraystretch}{1.2}
  \begin{tabular}{|c|c|c|c|}
   \hline
   $\alpha$ & Correct & Incorrect & No decision \\
   \hline
   $1.0$ & 0 (0.00\%) & 0 (0.00\%) & 10000 (100.00\%) \\
   \hline
   $0.1$ & 0 (0.00\%) & 0 (0.00\%) & 10000 (100.00\%) \\
   \hline
   $0.1^2$ & 8 (0.08\%) & 0 (0.00\%) & 9992 (99.92\%) \\
   \hline
   $0.1^3$ & 2053 (20.53\%) & 4 (0.04\%) & 7943 (79.43\%) \\
   \hline
   $0.1^4$ & 9692 (96.92\%) & 308 (3.08\%) & 0 (0.00\%) \\
   \hline
   $0.1^5$ & 9695 (96.95\%) & 305 (3.05\%) & 0 (0.00\%) \\
   \hline
   $0.1^6$ & 8020 (80.20\%) & 1980 (19.80\%) & 0 (0.00\%) \\
   \hline
   $0.1^7$ & 4031 (40.31\%) & 5969 (59.69\%) & 0 (0.00\%) \\
   \hline
  \end{tabular}
  \caption{Classification results (count / rate) of test data
  for uniformly setting $\alpha = 0.1^p$ ($0 \le p \le 7$).
  }
  \label{table: classification of test data}
 \end{center}
\end{table}

We present more precisely classification results 
on the case $\alpha = 0.1^5$.
Table~\ref{precise data for uniform alpha=0.1^5} presents
the predicted, correct, incorrect counts and rates 
for the each digit $0, 1, \ldots, 9$ in test data. 
Here, for label $\ell \in \{ 0, \ldots, 9 \}$, note that
\begin{itemize}
 \item Total: the number of $(x^k, y_k)$ satisfying $y_k = \ell$,
 \item Predicted: 
       the number of $(x^k, y_k)$ satisfying
       $\widehat{\Phi}_\Delta (x_k, \alpha) > 0$ and
       $\widehat{\Phi}_\Delta (x_k, \alpha) = \Phi_{\Delta, y_k} (x_k, \alpha)$
       (which includes the case $y_k \neq \ell$),
       that is, 
       our method returns the label $\ell$.
\end{itemize}
Note that in Table~\ref{precise data for uniform alpha=0.1^5}, 
the denominaters of the rates in this table uses the number
of ``Total'' for each digit. 
The sum of match and error number is total number of test data 
on label $\ell$.
One can find that our algorithm has a good performance 
for classification of every digit $0, 1, \ldots, 9$.
\begin{table}[htbp]
\begin{center}
 \renewcommand{\arraystretch}{1.2}
 \begin{tabular}{ccccc}
  True Digit & Total & Predicted & Correct & Incorrect \\
  \hline
  0 &  980 &  991 (101.12\%) &  972 (99.18\%) &  8 (0.82\%) \\
  1 & 1135 & 1177 (103.70\%) & 1130 (99.56\%) &  5 (0.44\%) \\
  2 & 1032 & 1012 ( 98.06\%) &  996 (96.51\%) & 36 (3.49\%) \\
  3 & 1010 & 1005 ( 99.51\%) &  970 (96.04\%) & 40 (3.96\%) \\
  4 &  982 &  969 ( 98.68\%) &  944 (96.13\%) & 38 (3.87\%) \\
  5 &  892 &  899 (100.78\%) &  860 (96.41\%) & 32 (3.59\%) \\
  6 &  958 &  967 (100.94\%) &  945 (98.64\%) & 13 (1.36\%) \\
  7 & 1028 & 1029 (100.10\%) &  987 (96.01\%) & 41 (3.99\%) \\
  8 &  974 &  938 ( 96.30\%) &  922 (94.66\%) & 52 (5.34\%) \\
  9 & 1009 & 1013 (100.40\%) &  969 (96.04\%) & 40 (3.96\%) \\
  \hline 
  Total & 10000 & 10000 (100.00\%) & 9695 (96.95\%) & 305 (3.05\%) \\
  \hline
 \end{tabular}	
 \caption{Numerical results of 
 predicted, correct and incorrect (count / rate)  
 for each label $\ell = 0, 1, \ldots, 9$
 by the baseline approach with $\alpha = 0.1^5$.
 }
 \label{precise data for uniform alpha=0.1^5}
\end{center} 
\end{table}

\par
\bigskip
\noindent
\textbf{Pointwise parameter algorithm.}

We next examine Algorithm~\ref{alg1}.
Recall that 
\[
 \widehat{\Phi}_\Delta (x^k, \alpha)
 = \max \{ \Phi_{\Delta, \ell} (x^k, \alpha) \mid
 \ell=0, 1, \ldots, 9 \},
\]
and $\Phi_{\Delta, \ell} (x^k, \alpha)$ are the numerically computed value
of $\widehat{\Phi} (x^k, \alpha)$ and $\Phi_\ell (x^k, \alpha)$,
respectively.
In Algorithm~\ref{alg1}, we find a \emph{pointwise}
appropriate $\alpha$ for each test data $(x^k, y_k) \in S$.
We emphasize that the aim of our algorithm is to compute
the characteristic function of the set $S_i$ defined in \eqref{S_i},
not to compute a function approximating the characteristic function
of $S_i$.
This is the reason why we compute a pointwise $\alpha$
for each $(x^k, y_k) \in S$.

Table~\ref{classification results by alg1}
presents the numerical results of the classifications
by Algorithm~\ref{alg1} with $\varepsilon = \varepsilon_*$
and $r = 0.1$.
One can find that the results by Algorithm~\ref{alg1}
is very close to the baseline approach.
However, the accuracy of total
in Table~\ref{precise data for uniform alpha=0.1^5}
is better than that in Table~\ref{classification results by alg1}.
For improving our method, we next investigate the distribution of $\alpha = 10^{-p}$, 
which is used to classify the handwritten digit
by Algorithm~\ref{alg1}. 
Now, let us set $\alpha = r^{\frac{p}{100}}$ with $r=0.1$
and examine the Algorithm~\ref{alg1} with $\varepsilon = \varepsilon_*$,
and find an optimal $p \in \mathbb{N} \cup \{ 0 \}$ to classify, 
that is,
\[
 \widehat{\Phi}_\Delta (x^k, r^{\frac{p}{100}}) > \varepsilon_*,
 \quad 
 \widehat{\Phi}_\Delta (x^k, r^{\frac{p-1}{100}}) \le \varepsilon_*.
\]
Figure~\ref{powermap: alg1} plots the power $q=p/100$ 
satisfying the above
for each datum $x^k \in S$.
One can find that the all power in Figure~\ref{powermap: alg1}
seems to be less than 4, and thus
the parameter $\alpha = 10^{-5}$ in Table~\ref{precise data for uniform alpha=0.1^5}
is quite smaller than those used in Algorithm~\ref{alg1}.
In a viewpoint of solving the heat equation,
it means that the baseline approach uses the solution of
the heat equation at a longer time than that
for Algorithm~\ref{alg1}.
Thus, 
the results in Table~\ref{precise data for uniform alpha=0.1^5}
can be attributed to the smoothing effect of the heat equation,
which happens to work favorably in this case.
\begin{table}[htbp]
\renewcommand{\arraystretch}{1.2}
 \begin{center}
  \begin{tabular}{cccccc}
   True & & & & & \\
   Digit & Total & Predicted & Correct & Incorrect & Avg. power \\
   \hline
   0 &  980 &  993 (101.33\%) &  973 (99.29\%) &  7 (0.71\%) & 3.931633 \\ 
   1 & 1135 & 1167 (102.82\%) & 1129 (99.47\%) &  6 (0.53\%) & 3.025551 \\
   2 & 1032 & 1008  (97.67\%) &  992 (96.12\%) & 40 (3.88\%) & 3.995155 \\
   3 & 1010 & 1008  (99.80\%) &  970 (96.04\%) & 40 (3.96\%) & 3.988119 \\
   4 &  982 &  971  (98.88\%) &  944 (96.13\%) & 38 (3.87\%) & 3.932790 \\
   5 &  892 &  904 (101.35\%) &  860 (96.41\%) & 32 (3.59\%) & 3.982063 \\
   6 &  958 &  962 (100.42\%) &  944 (98.54\%) & 14 (1.46\%) & 3.868476 \\
   7 & 1028 & 1037 (100.88\%) &  992 (96.50\%) & 36 (3.50\%) & 3.598249 \\
   8 &  974 &  939  (96.41\%) &  921 (94.56\%) & 53 (5.44\%) & 3.990760 \\
   9 & 1009 & 1011 (100.20\%) &  967 (95.84\%) & 42 (4.16\%) & 3.757185 \\
   \hline
   Total & 10000 & 10000 (100.00\%) & 9692 (96.92\%) & 308 (3.08\%) & 3.793500 \\
   \hline
  \end{tabular}
  \caption{Numerical results (count / rate) of classifications by
  Algorithm~\ref{alg1}.
  Average power is the average of $q \in \mathbb{N}$ of $\alpha = 10^{-q}$
  used in the classification of test data.
  }
  \label{classification results by alg1}
 \end{center}
\end{table}

\begin{figure}[htbp]
 \begin{center}
  \includegraphics[scale=0.5]{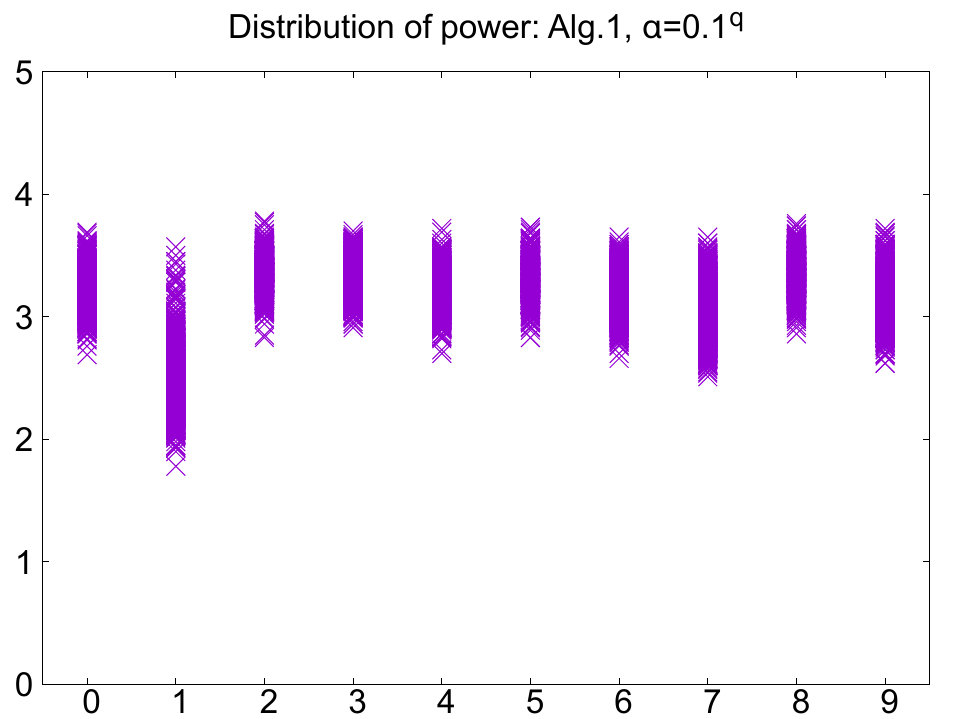}
  \caption{Distribution of the power $q=p/100$ 
  of $\alpha = 10^{-q}$ using the classification 
  by Algorithm~\ref{alg1} with $\varepsilon = \varepsilon_*$.
  }
  \label{powermap: alg1}
 \end{center} 
\end{figure}

For demonstration the accuracy of our method 
we compare it with the accuracy obtained by applying 
the support-vector machine, which is given here. 
By using the support vector machine with the activation function, 
we can classify in the following way: 
Let $\varphi: {\mathbb R}^{10} \to  {\mathbb R}^{10}$ 
be an activation function, a matrix $M \in {\mathbb R}^{d \times 10}$ 
and a bias $b \in {\mathbb R}^{10}$. 
In this formulation $\varphi_\ell$ indicates the propability 
such that the datum $x$ corresponds to $\ell$, 
where $\varphi_\ell$ is a component of $\varphi$, namely, 
$\varphi = (\varphi_0, \varphi_1, \ldots, \varphi_9)$. 
Also, to get suitable $M$ and $b$ we define a loss function $f = f(M, b)$ by 
\[ f(M, b) = \sum_{\ell = 0}^9 \sum_{k=1}^{K_\ell} | \varphi( M p^k_\ell + b) - \ell|^2. 
\] 
After finding a pair $(M_*, b_*)$ such that 
$f(M_*, b_*) = \min \{ f(M, b) \mid M \in {\mathbb R}^{d \times 10}, b \in {\mathbb R}^{10} \}$, 
we classify a testing datum $x \in {\mathbb R}^d$ as $\ell_*$ if
\[ 
 \varphi_{\ell_*} (M_* x + b_*) 
 = \max\{ \varphi_\ell (M_* x + b_*) \mid \ell = 0, 1, \ldots, 9 \}. 
\] 

In this article, we find $(M_*, b_*)$ by the gradient descent 
and investigate the accuracy of the support vector machine 
for each of three types of activation functions, 
the Softmax function 
$\varphi_i (\xi) = \exp (\xi_i) / (\sum_{\ell = 0}^9 \exp(\xi_\ell))$, 
the RELU function 
$\varphi_i (\xi) = \max\{0, \xi_i\}$, and the linear function 
$\varphi_i(\xi) = \xi_i$ for $\xi = (\xi_0, \xi_1, \ldots, \xi_9) \in {\mathbb R}^{10}$ 
and  $i = 0, 1, \ldots, 9$. 
Here, Table \ref{compare: diff vs svm} shows the comparison
between the support vector machine with several activation functions 
and our method.
%
%
\begin{table}[htbp]
\renewcommand{\arraystretch}{1.2}
\begin{tabular}{p{65pt}|ccccc}
\hline
Methods & Baseline approach & Algorithm~\ref{alg1} & Softmax  &  RELU & Linear  \\
\hline
Error rate (\%) & 3.05 & 3.08 & 7.71  & 8.76 & 13.8 \\
\hline \end{tabular}
\caption{Error rates by methods} 
\label{compare: diff vs svm}
\end{table}
%
%

In results shown by Table \ref{compare: diff vs svm} the epoch number is 100 
for the all three ways of (SVM), since the error rates increased 
for the epoch number greater than 100. 
Here, we emphasise that in Algorithm 1 the parameter $\alpha$ is 
determined independently of the labels of the testing data. 
On the other hand, the epoch number must be defined, 
after we know the labels of the test data. 
Thus, the most important advantage of our classification methods is 
to achieve high correct rate without the labels of the testing data. 

We next investigate the accuracy of our approach
for the choice of the parameter $\varepsilon$.
Table~\ref{alg1: comparison of epsilon} shows the counts of correct
of the classifications by Algorithm~\ref{alg1} with
$\varepsilon = \varepsilon_*$, $10^{-200}$, $10^{-100}$, $10^{-50}$,
$10^{-20}$, and $10^{-10}$.
The counts of correct are the almost same when $\varepsilon < 10^{-20}$,
and accuracy of classification become worse when $\varepsilon^{-10}$.
The computational time increases when $\varepsilon$ increases.
Hence, the choice of the machine epsilon $\varepsilon = \varepsilon_*$
is reasonable in this case, although the count of correct 
in our numerical results attains its maximum 
when $\varepsilon = \varepsilon^{-50}$.
\begin{table}[htbp]
\renewcommand{\arraystretch}{1.2}
 \begin{center}
  \begin{tabular}{cc|cccccc}
   \hline
   True & & 
   \multicolumn{6}{c}{Correct (count)}
   \\
   \cline{3-8}
   Digit & Total & $\varepsilon_*$ & $10^{-200}$ 
   & $10^{-100}$ & $10^{-50}$ & $10^{-20}$ & $10^{-10}$ \\
   \hline
   0 &  980 &  973 &  973 &  973 &  972 &  972 &  968 
   \\ 
   1 & 1135 & 1129 & 1129 & 1129 & 1130 & 1130 & 1132 
   \\
   2 & 1032 &  992 &  992 &  992 &  996 &  996 &  959 
   \\
   3 & 1010 &  970 &  970 &  971 &  970 &  970 &  957 
   \\
   4 &  982 &  944 &  944 &  943 &  946 &  944 &  941 
   \\
   5 &  892 &  860 &  860 &  859 &  860 &  860 &  823 
   \\
   6 &  958 &  944 &  944 &  944 &  945 &  945 &  938 
   \\
   7 & 1028 &  992 &  992 &  990 &  989 &  987 &  986 
   \\
   8 &  974 &  921 &  921 &  923 &  922 &  922 &  877 
   \\
   9 & 1009 &  967 &  967 &  968 &  968 &  969 &  964 
   \\
   \hline
   Total & 10000 & 9692 & 9692 & 9692 & 9698 & 9695 & 9545 
   \\
   \hline
  \end{tabular}
  \caption{Comparison of the counts of correct 
  with respect to $\varepsilon$ in Algorithm~\ref{alg1}.
  }
  \label{alg1: comparison of epsilon}
 \end{center}
\end{table}

Moreover, we investigate the relation between 
the number of training data and accuracy of the classification.
In this examinations, we pick up 
just a subset of training data of the form
\[
 A (K) = \{ (\xi^k, \eta_k) \in A \mid k=1, 2, \ldots, K \}
\]
for $K=5000, 10000, 15000, \ldots, 60000$
to make the discreminants $\Phi_{\ell, \triangle} (x, \alpha)$,
$\ell=0, 1, 2, \ldots, 9$.
Table~\ref{relation: tr_num and match} compares 
the counts of correct by the baseline approach and
Algorithm~\ref{alg1} when we make every $\Phi_\ell (x, \alpha)$
with restricted training data $A(K)$.
One can find that our algorithms establish an accuracy of over $93\%$
even if we make the discreminants just by
$5000$ training data.
\begin{table}[htbp]
\renewcommand{\arraystretch}{1.2}
 \begin{center}
  \begin{tabular}{c|c|c}
   & \multicolumn{2}{c}{Correct (count)} \\
   \cline{2-3}
   Num. of training data & baseline approach 
       & Algorithm~\ref{alg1} \\
   ($K=\sum_{i=0}^9 K_i$)
   & ($\alpha = 10^{-5}$) & ($\varepsilon = \varepsilon_*$, $\alpha = 10^{-p}$) \\ 
   \hline
    5000 & 9354 & 9343 \\
   10000 & 9475 & 9463 \\
   15000 & 9539 & 9523 \\
   20000 & 9569 & 9555 \\
   25000 & 9602 & 9592 \\
   30000 & 9622 & 9617 \\
   35000 & 9644 & 9637 \\
   40000 & 9660 & 9653 \\
   45000 & 9666 & 9658 \\
   50000 & 9669 & 9667 \\
   55000 & 9687 & 9682 \\
   60000 & 9695 & 9692 \\
   \hline
  \end{tabular}
  \caption{Relation between number of training data and
  counts of correct for the baseline approach and 
  Algorithm~\ref{alg1}. 
  }
  \label{relation: tr_num and match}
 \end{center}
\end{table}

We also give a numerical result of the classification 
with online learning by Algorithm~\ref{alg3}.
In this examination, 
thanks to discreminants 
$\Phi_\ell (x,\alpha)$ ($\ell = 0, 1, \ldots, 9$)
made by 60000 training data, 
we execute the following for all test data 
$x^k$ ($k=1, 2, \ldots, 10000$):
\begin{enumerate}
 \item Classify the datum $x^k$ by $\Phi_\ell (x, \alpha)$
       to obtaine the classified label $i(x^k)$.
 \item Update the discreminants $\Phi_\ell (x, \alpha)$
       with the obtained label $i(x^k)$ or the correct label $y_k$
       of $x^k$ by the following rule.
       \begin{description}
	\item[with supervisor] Replace $\Phi_{y_k} (x, \alpha)$
		   by $\Phi_{y_k} (x, \alpha) + \exp (\alpha |x - x^k|)$.
	\item[without supervisor] Replace $\Phi_{i(x_k)} (x, \alpha)$
		   by $\Phi_{i(x_k)} (x, \alpha) + \exp (\alpha |x - x^k|)$.
       \end{description}
\end{enumerate}
Table \ref{alg3: numerical result of online learning}
presents the numerical results of the classification by 
Algorithm~\ref{alg3} with or without supervisor.
One can find not only the algorithm with supervisor refines
the classification results by Algorithm~\ref{alg1},
but also the algorithm without supervisor establishes
9686 counts of correct, which is just $6$ less than 
that by Algorithm~\ref{alg1}.
These results indicate that our algorithm is robust to 
learning from incorrectly predicted data.
\begin{table}[htbp]
\renewcommand{\arraystretch}{1.2}
 \begin{center}
  \begin{tabular}{cc|c|c|c}
   & & 
   \multicolumn{3}{c}{Correct (count)}
   \\
   \cline{3-5}
   True & & & 
   \multicolumn{2}{c}{Algorithm~\ref{alg3}} \\
   \cline{4-5}
   Digit & Total & Algorithm~\ref{alg1} & with supervisor & without supervisor \\
   \hline
   0 &  980 &  973 &  972 &  972 
   \\ 
   1 & 1135 & 1129 & 1130 & 1129 
   \\
   2 & 1032 &  992 &  995 &  993 
   \\
   3 & 1010 &  970 &  974 &  971 
   \\
   4 &  982 &  944 &  950 &  945 
   \\
   5 &  892 &  860 &  864 &  861 
   \\
   6 &  958 &  944 &  944 &  944 
   \\
   7 & 1028 &  992 &  995 &  989 
   \\
   8 &  974 &  921 &  925 &  922 
   \\
   9 & 1009 &  967 &  972 &  960 
   \\
   \hline
   Total & 10000 & 9692 & 9721 & 9686 
   \\
   \hline
   \multicolumn{2}{c|}{Avg. power}
   & 3.7935 & 3.7859 & 3.7859 \\
   \hline
  \end{tabular}
  \caption{Comparison of counts of correct
  between Algorithm~\ref{alg1} and 
  Algorithm~\ref{alg3} with or without supervisor.
  In Algorithm~\ref{alg3}, 
  ``with supervisor'' means that the label $i(p_m)$
  of the test data $p_m$ is given as the correct label.
  On the other hand, 
  ``without supervisor'' means that $i(p_m)$
  of the test data $p_m$ is given as the 
  predicted label by Algorithm~\ref{alg3}
  with training data $\bigcup_{j=0}^9 A_j$ and $p_1, \ldots, p_{m-1}$.}
  \label{alg3: numerical result of online learning}
 \end{center}
\end{table}

Finally, we conclude this section presenting
the numerical results of classification with noises.
Let the training data \eqref{mnist training data} 
by MNIST be given as
\[
 \xi^k = (\xi^k_{i,j})_{1 \le i,j \le 28}, \quad
 \xi^k_{i,j} \in [0, 255] \cap \mathbb{Z}
 \quad \mbox{for} \ i,j \in \{ 1, 2, \ldots, 28 \}. 
\]
We now add a uniform random noise of the form
\[
 \zeta^k = (\zeta^k_{i,j})_{1 \le i,j \le 28}, \quad
 \zeta^k_{i,j} \in [-L,L] \cap \mathbb{Z}
\]
with a fixed constant $L \in \mathbb{N}$ to $\xi^k$,
which is of the form
\[
 \tilde{\xi}^k = (\tilde{\xi}^k_{i,j})_{1 \le i,j \le 28}, \quad
 \tilde{\xi}^k_{i,j} = \max \{ 0, \min \{ 255, \xi^k_{i,j} + \zeta^k_{i,j} \} \}.
\]
See Figure \ref{org and noisy} for details of an original datum
and the datum with added noise.
We also prepare test data with noise $\tilde{x}^k$ of \eqref{mnist test data}
as the similar manner.
\begin{figure}[htbp]
 \begin{center}
  \includegraphics[scale=0.4]{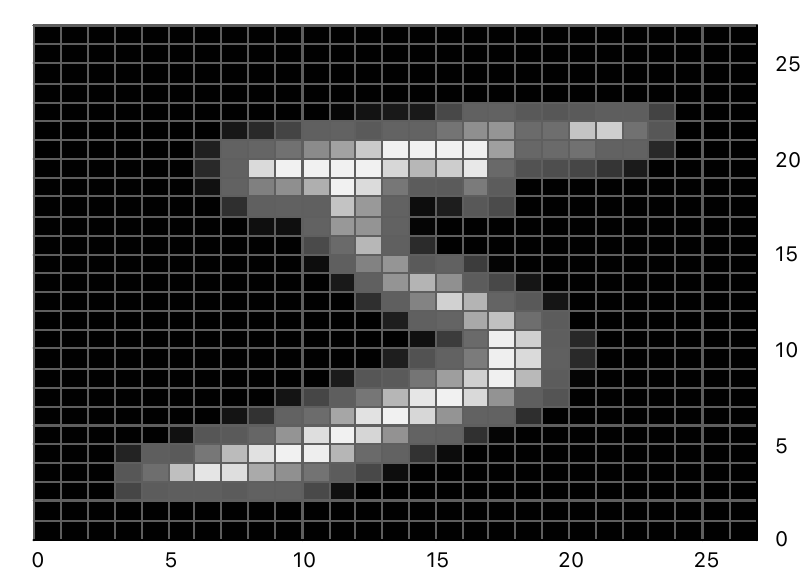}
  \includegraphics[scale=0.4]{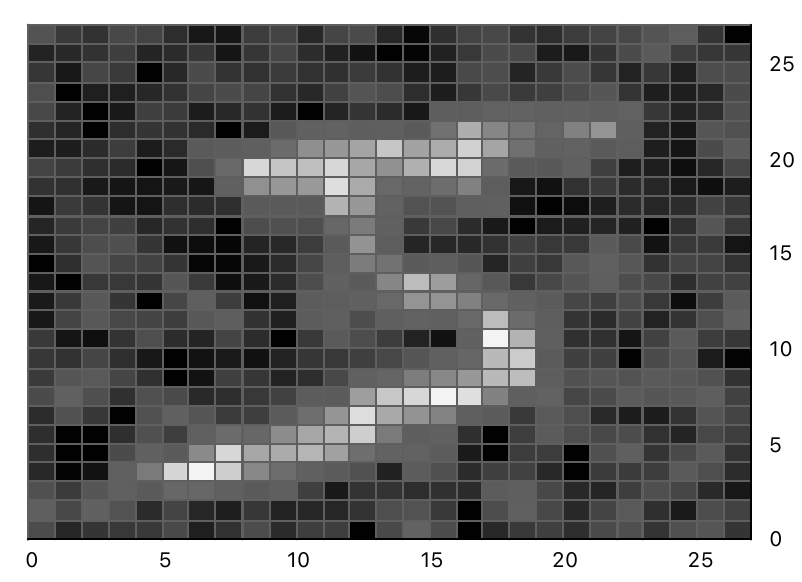}
  \caption{An original datum by MNIST 
  and the datum with added noise ($L=100$).}
  \label{org and noisy}
 \end{center}
\end{figure}
For fixed amplitude $L$ of the random noise,
we execute 10 times of experiments of the classification  
by Algorithm~\ref{alg1}. 
Table \ref{noise test} shows the minimum, maximum,
average and sample variance of counts of correct
of the classification results by the training and test data
with added noise.
Our approach is also robust to 
the classification with the data with added noise.
\begin{table}[htbp]
\renewcommand{\arraystretch}{1.2}
 \begin{center}
  \begin{tabular}[t]{|c|c|c|c|c|}
   \hline
   & \multicolumn{4}{|c|}{Correct (count)} \\
   \cline{2-5} 
   $L$ & Minimum & Maximum & Average & Variance \\
   \hline
   50 & 9650 & 9669 & 9660.7 & 47.79 \\
   60 & 9632 & 9657 & 9647.2 & 52.89 \\
   70 & 9625 & 9652 & 9641.4 & 78.71 \\
   80 & 9607 & 9639 & 9622.4 & 121.60 \\
   90 & 9584 & 9633 & 9604.0 & 253.11 \\
   100 & 9553 & 9596 & 9577.6 & 148.93 \\
   \hline
  \end{tabular}
  \caption{Summary of 10 times classification results by
  the training and test data with added noise.}
  \label{noise test}
 \end{center}
\end{table}

\section{Conclusion and discussion}

We extended the diffusive method proposed by \cite{GGOU:2013CPAA}
to $N$-class data classification.
To handle more than 3 classes classification, we consider 
a solution $(u_1, \ldots, u_N)$ of a system of the heat equations 
whose initial datum is the characteristic function 
$(\chi_{A_1}, \ldots, \chi_{A_N})$ of a given family 
$\{ A_i \mid i=1, \ldots, N \}$ of training data.
By employing the solutions, the classification of test datum $x$
can be established by investigating the diffusive signs 
$\lim_{t \to 0} \mathrm{sgn} (u_i (x, t) - u_j (x, t))$
for $i, j=1, \ldots, N$. 
At the practical stage, we investigate the sign of
$u_i (x, h) - u_j (x, h)$ for $i, j = 1, \ldots, N$
at a small time $h \ll 1$.
By splitting the solution of the heat equation
into those for each label $\ell = 1, \ldots, N$,
our method enables us to consider different diffusion equaitons
for each label $\ell$.
In this paper we just consider the situation such that
the system of heat equations has different diffusion coefficients
for each label $\ell$. 
We obtain a circular classification results characterized
by the essential distance from the training data set $A_i$.
See Figure \ref{fig: result-3-4_triple}
for numerical results comparing the cases where 
the diffusion coefficients are different 
and where they are equal in 3-class classification.

We applied our approach to the classification of hand written digits
proposed by MNIST database \cite{MNIST_Url,MNIST}.
When we apply our method, this database has two issues:
training and test data are discrete, and
the dimension of the data ($28 \times 28$) is too large to compute.
We settled the first issue by using a sum of heat kernels, 
in which the origin is moved to each training data,
instead of solving the heat equations.
However, directly evaluating the heat kernel formula
in such a high-dimensional feature space often leads
to numerical overflow or underflow.
To address this issue, we used the normalized formula
of the heat kernel limiting the numerical issue to underflow only.
As a result, we established a classification algorithm
by finding a pair of the label and the time parameter 
which is the first to emerge from underflow.
We call this idea ``pointwise parameter algorithm'' in this paper.
On the other hand, we investigate the ``baseline approach'',
which finds a uniform time parameter avoiding numerical underflow
for all training data.

In our numerical experiments, we compare the numerical results
of the baseline approach, which uses a single time parameter
derived from the training dataset for all inputs, 
and the pointwise parameter algorithm, which employs 
input-independent time parameters for classification.
Our approaches attained a high accuracy of approximately 97\% 
despite the absence of preprocessing and data augumentation,
indicating the robustness of the method on raw handwritten digit images.
Moreover, it remains robust even when trained on a small
subset of the training data;
for example, using 5000 samples out of the original 60,000,
the classifier still archives over 93\% accuracy.
In addition, it is robust to the training data
that include misclassified samples or noisy data.

\bibliographystyle{plain}
\bibliography{ref_diffsign}
 
\end{document}